\documentclass[a4paper,11pt,fleqn]{article}

\usepackage[utf8]{inputenc}
\usepackage[T1]{fontenc}
\usepackage{textcomp}
\usepackage{amsmath,amssymb,amsthm,mathrsfs}
\usepackage{lmodern}
\usepackage[a4paper]{geometry}
\usepackage{graphicx}
\usepackage[usenames,dvipsnames]{xcolor}
\usepackage{microtype}
\usepackage{enumerate}
\usepackage{hyperref}
\usepackage{authblk}
\hypersetup{pdfstartview=XYZ}
\newtheoremstyle{theoreme}
 {\topsep} 
 {\topsep} 
 {} 
 {0pt} 
 {\bfseries} 
 {\newline} 
 {3pt} 
 {} 
\newtheoremstyle{proposition}
{\topsep} 
 {\topsep} 
 {} 
 {0pt} 
 {\bfseries} 
 {\newline} 
 {3pt} 
 {} 
 \newtheoremstyle{lemma}
 {\topsep} 
 {\topsep} 
 {} 
 {0pt} 
 {\bfseries} 
 {\newline} 
 {3pt} 
 {} 
\newtheoremstyle{definition}
 {\topsep}
 {\topsep}
 {}
 {0pt}
 {\bfseries}
 {\newline}
 {3pt}
 {}
\newtheoremstyle{remarque}
 {\topsep}
 {\topsep}
 {}
 {0pt}
 {\bfseries}
 {: }
 {3pt}
 {} 
\theoremstyle{theoreme}
\newtheorem{Thm}{Theorem}[section]
\newtheorem*{Thm*}{Theorem}
\newtheorem*{Notation}{Notation}
\newtheorem{Ex}[Thm]{Example}
\theoremstyle{lemma}
\newtheorem{Lem}[Thm]{Lemma}
\theoremstyle{proposition}
\newtheorem{Prop}[Thm]{Proposition}
\newtheorem*{Prop*}{Proposition}
\newtheorem{Cor}[Thm]{Corollary}
\theoremstyle{definition}
\newtheorem{Def}[Thm]{Definition}
\theoremstyle{remarque}
\newtheorem{Rem}[Thm]{Remark}
\newtheorem{Question}{Question}

\newtheorem*{Claim*}{Claim}

\newcommand\N{\mathbb{N}}
\newcommand\R{\mathbb{R}}
\newcommand\Z{\mathbb{Z}}
\newcommand\Q{\mathbb{Q}}

\newcommand\comment[1]{}
\newcommand\valuedfield{K\overset{v}{\longrightarrow}G\overset{v_G}{\longrightarrow}\Gamma}
\newcommand\eqrel{\sim_{\phi}}
\newcommand\qophi{\precsim_{\phi}}

\newcommand\ringv{\mathcal{O}_v}
\newcommand\ringw{\mathcal{O}_w}
\newcommand\unitsv{\mathcal{U}_v}
\newcommand\unitsw{\mathcal{U}_w}
\newcommand\idealv{\mathcal{M}_v}
\newcommand\idealw{\mathcal{M}_w}

\newcommand\dom{\text{Dom}}

\newcommand\constants{\mathcal{C}}
\newcommand\supp{\text{supp}}
\newcommand\Hahn{\text{H}}

\newcommand\qohom{\precsim_{\hat{\omega}}}
\newcommand\quotientval{\frac{v}{w}}

\newcommand\compv{(B_{v_G})_{\gamma}}
\newcommand\grouppsiplus{(G_{\psi})^{\lambda}}
\newcommand\grouppsiminus{(G_{\psi})_{\lambda}}
\newcommand\comppsi{(B_{\psi})_{\lambda}}
\newcommand\groupsigplus{G_{\sigma}(f,+)}
\newcommand\groupsigminus{G_{\sigma}(f,-)}
\newcommand\compsig{H_{\sigma}(f)}
\newcommand\divhull[1]{\widehat{#1}}

\begin{document}
\title{The differential rank of a differential-valued field}
\author[1]{Salma Kuhlmann}
\author[1]{Gabriel Lehéricy\footnote{email: gabriel.lehericy@uni-konstanz.de}}
\affil[1]{Fachbereich Mathematik und Statistik, Universität Konstanz, 78457 Germany}

\maketitle
\begin{abstract}
  We develop a notion of  (principal) differential rank for differential-valued fields, in analog of the exponential rank developed in \cite{Kuhlmann} 
  and of the difference rank developed in \cite{KuhlmannPointMatusinski}. We give several characterizations of this rank. 
   We then give a method to define a derivation on a field of generalized power series and use this method to show 
that any totally ordered set can be realized as the 
principal differential rank of a H-field.\\

\emph{keywords:} Valuation, differential field, differential-valued field, H-field, asymptotic couple, generalized power series\\

 \end{abstract}

\section*{Introduction}  
  The rank of a valued field (i.e the order-type of the set of coarsenings of the given valuation, ordered by inclusion, see
\cite{EnglerPrestel} and \cite{ZariskiSamuel}) 
is an important invariant. It has three equivalent characterizations: 
one at the level of the valued field $(K,v)$ itself, another one at the level of the value group $G:=v(K^{\times})$ and a third one at the level of 
the value chain $\Gamma:=v_G(G^{\neq0})$ of the value group. Recently, notions of ranks have been developed for valued 
  fields endowed with an operator; examples of this are the exponential rank of an ordered exponential field (see \cite[Chapter 3, Section 2]{Kuhlmann})
  and the difference rank of a valued difference field (see \cite[Section 4]{KuhlmannPointMatusinski}). 
  In this paper, we are 
  interested in pre-differential-valued fields as introduced by Aschenbrenner and v.d.Dries in \cite{Aschdries2}, i.e valued fields endowed with 
  a derivation which is in some sense compatible with the valuation (see Section \ref{differentialranksection} below for the precise 
definition). 
  We will pay a special attention to the class of H-fields (see Definition \ref{Hfielddef}) introduced by Aschenbrenner and v.d.Dries in \cite{Aschdries2}
  and \cite{Aschdries3}. H-fields are a generalization of Hardy fields (introduced by Hardy in \cite{Hardy}). H-fields are particularly interesting because of the central role they play in the theory
of transseries and in the model-theoretic study of Hardy fields (see \cite{DriesAschHoev}). A substantial part of our work relies on the theory of asymptotic couples. 
 Asymptotic couples were introduced by Rosenlicht in \cite{Rosenlicht} and later studied by Aschenbrenner and
   v.d.Dries in \cite{Aschdries} and also play a central 
   role in the study of H-fields and transseries (see \cite{Aschdries2},\cite{Aschdries3} and \cite{DriesAschHoev}). 
   They consist of an ordered abelian group $G$ together with a map $\psi: G^{\neq0}\to G$ satisfying certain properties 
   (see Section \ref{acsubsection} below) and appear naturally as the value group of pre-differential-valued fields.

  \comment{The first goal of this paper  is to introduce and study 
  the notion of differential rank for pre-differential-valued fields (that is, the order type of the set of coarsenings 
  of a given valuation which are compatible with the logarithmic derivative)
  in analog of the already existing notions of rank.
  We do this by first defining the general notions of the  ``$\phi$-rank'' and of ``principal $\phi$-rank'' of a valued field 
endowed with an arbitrary operator $\phi$ (see Section \ref{phiranksection} below) which 
  generalizes at the same time the definitions of the (principal) exponential rank of an ordered exponential field and the definition 
  of the (principal) difference rank of a  valued  difference field. This allows us to define the (principal) differential rank of a valued differential field $(K,v,D)$ as the (principal) $\phi$-rank of the valued 
  field $(K,v)$ where $\phi$ is  defined as the logarithmic derivative. }

The first goal of this paper is 
to  introduce and study the differential rank of a differential-valued field. 
We start by defining the general notions of  ``$\phi$-rank'' and of ``principal $\phi$-rank'' of a valued field 
endowed with an arbitrary operator $\phi$ (see Definition \ref{phirankdef} below). This allows us to define the (principal) differential rank of a 
  pre-differential-valued field $(K,v,D)$ as the (principal) $\phi$-rank of the valued 
  field $(K,v)$ where $\phi$ is  defined as the logarithmic derivative.
  In Theorem \ref{diffrankthreelevelsprop}, we show that the differential rank 
of a pre-differential-valued field $(K,v,D)$ is equal to the $\psi$-rank (see Definition \ref{phirankdef}) of the asymptotic couple $(G,\psi)$ associated to $(K,v,D)$, and that it is also equal to the 
$\omega$-rank of the value chain $\Gamma$ if $K$ is a H-field, where $\omega:=\psi_{\Gamma}$ is the map induced by $\psi$ on $\Gamma$ (see Definition \ref{definducedmap}). 
We then characterize the differential rank via 
 the maps $\phi$, $\psi$ and $\omega$ (see Theorem \ref{diffrankcharbyqosThm}). 
 
We want to point out that the notion of (principal) $\phi$-rank defined in Section \ref{phiranksection} generalizes
simultaneously that of the (principal) exponential rank defined in \cite{Kuhlmann} and that
  of the (principal) difference rank defined in \cite{KuhlmannPointMatusinski}. 
  \comment{Theorems \ref{diffrankthreelevelsprop} and \ref{diffrankcharbyqosThm} of this paper are analogs 
  of the results on the exponential rank and on the difference rank in \cite{Kuhlmann} and \cite{KuhlmannPointMatusinski}.}
  Theorem \ref{diffrankthreelevelsprop} is the analog of
    \cite[Theorem 3.25]{Kuhlmann} and of  \cite[Theorem 4.7]{KuhlmannPointMatusinski}, and 
   Theorem \ref{diffrankcharbyqosThm}  corresponds to  \cite[Theorem 3.30]{Kuhlmann} and to 
    \cite[Theorem 5.3, Corollary 5.4 and Corollary 5.5]{KuhlmannPointMatusinski}.
 However, we note that the exponential rank and the difference rank both have a characterization 
in terms of residue fields: a coarsening $w$ of $v$ lies in the difference rank of $(K,v,\sigma)$ if and only if $\sigma$ induces an automorphism on the residue 
field $Kw$, and a similar result holds for the exponential rank. It turned out that this characterization does not go through to the differential case; although 
the map induced by the derivation on $Kw$ does play a role in determining whether $w$ is compatible with the logarithmic derivative, it is not enough to characterize the differential rank. 
 Theorem \ref{characterizationofdiffrankThm} gives a full characterization of the valuations $w$ which lie in the differential rank.

  Our study of the differential rank revealed a problem  which does not arise in the exponential case nor in the difference case. 
  This problem is related to the existence of ``cut points'' in asymptotic couples, which is a notion we introduce in Section \ref{acsubsection} (Definition \ref{cutpointdef}). A cut point of an asymptotic couple 
$(G,\psi)$ is an element $c$ of $G$ such that $\psi(c)$ is archimedean-equivalent to $c$.
If $(G,\psi)$ is an asymptotic couple, then the $\psi$-rank of $G$ is too coarse to give a satisfactory description of the behavior of $\psi$. Indeed, 
 if $c$ is a cut point then any $\psi$-compatible convex subgroup of $G$ must contain $c$, which means that the $\psi$-rank of $G$ does not give any information on the behavior 
of $\psi$ between $c$ and $0$.  This lead us to introduce the notion of (principal) unfolded differential rank. The unfolded differential rank is obtained by 
considering a family $(\psi_g)_{g\in G}$ of translates of $\psi$. As $g$ goes to $0$, the cut points of $\psi_g$ also go to $0$, which means that 
the $\psi_g$-rank of $G$ gives more information about $\psi$ than the $\psi$-rank, provided $g$ is small enough. We then define the unfolded differential rank of a pre-differential-valued field with asymptotic couples $(G,\psi)$ 
as 
the union of all the $\psi_g$-ranks of $G$ with $g$ ranging in $G^{<0}$. The unfolded differential rank contains the differential rank, and in general this inclusion is strict.
We show  that the unfolded differential rank is closely related to the notion of exponential rank, which makes it an interesting object. In fact, 
Corollary \ref{unfoldedrankisexponentialrank} says that we can view the unfolded differential rank as a sort of generalization of the exponential rank for pre-differential-valued fields which do not necessarily admit 
an exponential.

  It was shown in \cite[Corollary 5.9]{Kuhlmann} (respectively, in \cite[Corollary 5.6]{KuhlmannPointMatusinski}) that any linearly ordered set can be realized as the principal exponential rank 
  (respectively, the principal difference rank)
  of 
  an exponential field (respectively, a difference field). 
The second goal of this paper is to state a similar result for the differential case. 
In doing so, we were confronted with  the following question:

\begin{Question}\label{question1}
   Given an ordered abelian group $G$ and a field $k$ of characteristic $0$, under which conditions on $k$ and $G$ can we define a derivation on the field
  of generalized power series $k((G))$ making it a differential-valued field? a H-field?
  \end{Question}
 
 Such a  problem has already been 
  studied by the first author and Matusinski in \cite{KuhlmannMatusinski} and \cite{KuhlmannMatusinski2}. In \cite{KuhlmannMatusinski}, they
  considered a field $K$ of generalized power series whose value group $G$ is a Hahn product of copies of $\R$ over a given linear order 
  $\Phi$ (see Section \ref{powerseries} for the definition of Hahn product).  Assuming 
  that a derivation is already defined on the chain $\Phi$ of fundamental monomials, they gave conditions for this derivation to be extendable to the group $G$ of monomials and 
  to the whole field $K$. They then proceeded to give conditions for this derivation to be of Hardy type. The present paper extends the results of \cite{KuhlmannMatusinski}, giving a full answer 
  to Question \ref{question1} in the case where $D$ is required to be a $H$-derivation (see Section \ref{powerseries} for the definition of 
$H$-derivation) and constructing $D$ explicitly.
Our approach uses asymptotic couples. We will first answer the following question (see Section \ref{acsubsection} for the definition of H-type asymptotic couple):
  
  \begin{Question}\label{question}
    Let an $H$-type asymptotic couple $(G,\psi)$ and a field $k$ of characteristic $0$ be given. Under which condition on $(G,\psi)$ and $k$ can we define 
   a derivation $D$ on $K:=k((G))$ such that $(K,v,D)$ is a differential-valued field (or a H-field) 
   whose associated asymptotic couple is
   $(G,\psi)$?
   \end{Question}
 
   We give a full answer to Question \ref{question} in Section \ref{derivationfromacsection} (Theorem \ref{mainThm}). We then use 
  Theorem \ref{mainThm} to answer Question \ref{question1}. Our answer to Question \ref{question1} involves the notion of shift (see Section \ref{Hderivationsection} 
  for a definition of this notion). 
 A connection between Hardy-type derivations  and shifts was already established in \cite{KuhlmannMatusinski} and \cite{KuhlmannMatusinski2}, where the authors showed that 
 certain shifts on the value chain $\Gamma$ of a field of generalized power series $K$ can be lifted to a Hardy-type derivation on $K$.  We basically 
 show in this paper that any H-derivation on a field of power series comes from a shift satisfying certain properties 
(Theorems \ref{answertoquestion1} and \ref{HardytypeThm}).

  The paper is organized as follows. It starts with a preliminary section which sets the notation and recalls a few basic facts about 
   valuation theory, fields of generalized power series and quasi-orders. Section \ref{phiranksection} introduces the general notion of the $\phi$-rank of a valued
  field endowed with an operator $\phi$.  We show in Proposition \ref{fieldgroupchain} that, provided $\phi$ satisfies some reasonable condition, the $\phi$-rank of a valued field can be characterized at three levels: at the level
  of the field itself, 
  at the level of the value group and the level of the value chain of the value group. We also show that, if 
$\phi$ happens to be an increasing map of a quasi-ordered set, then it naturally induces a quasi-order which is helpful for the 
characterization of the $\phi$-rank (see Proposition \ref{phiqoprop}). This gives us a generalization of \cite[Corollary 5.4 and 5.5]{KuhlmannPointMatusinski}
(one can recover the results of \cite{KuhlmannPointMatusinski} by setting 
$A:=\Gamma$ and $\phi:=\sigma_{\Gamma}$ and applying Proposition \ref{phiqoprop}). 
Section \ref{acsubsection} is dedicated to the study of asymptotic couples. We introduce
the notion of cut point 
(see Definition \ref{cutpointdef}), and then use this notion to  describe the behavior of the map $\psi$ 
of an asymptotic couple.  The results of Section \ref{acsubsection} will be  useful for the study of the differential rank. 
Section \ref{differentialranksection} is dedicated to the study of the differential rank of a pre-differential-valued field. 
In Section \ref{diffranksubsec}, we apply Section \ref{phiranksection} 
to the particular case where $\phi$ is the logarithmic derivative of a pre-differential-valued field. Applying Proposition \ref{fieldgroupchain}, 
we obtain Theorem \ref{diffrankthreelevelsprop}, which states that the differential rank is characterized at the level of the asymptotic couple associated to the field. 
We then apply Proposition \ref{phiqoprop} to the case of H-field, and we obtain a characterization of the differential rank by a quasi-order induced by 
the logarithmic derivative (Theorem \ref{diffrankcharbyqosThm}). Finally, we characterize  the coarsenings $w$ of $v$ which lie in the 
differential rank of $(K,v,D)$ (Theorem \ref{characterizationofdiffrankThm}). In Section \ref{unfrksection}, we introduce the notion of unfolded differential rank. 
We characterize the coarsenings $w$ of $v$ which lie in the unfolded differential rank of $(K,v,D)$ (Theorem \ref{charofunfrk}) and show that 
the unfolded differential rank coincides with the exponential rank if 
$K$ is equipped with an exponential (Corollary \ref{unfoldedrankisexponentialrank}). Section \ref{powerseries} is dedicated to the construction of  a derivation on a field of generalized power series. 
In Section \ref{derivationfromacsection}, we show a method to construct a derivation on a field $k((G))$ of generalized power series from a H-type asymptotic couple $(G,\psi)$.
We answer Question \ref{question} in Theorem \ref{mainThm}, giving a necessary and sufficient condition on $(G,\psi)$ for the existence of a derivation on 
the field of power series $K:=k((G))$ that makes $K$ a differential-valued field of asymptotic couple $(G,\psi)$. We actually define the derivation explicitly. 
We then answer Question \ref{question1} for the case where $D$ is required to be a H-derivation in Section \ref{Hderivationsection} (Theorem \ref{answertoquestion1}).
We characterize the existence of a H-derivation on $K$ by the existence 
of a certain shift on the value chain $\Gamma$ of $K$.
 In Section 
\ref{Hardytypesection} we answer a variant of Question \ref{question1} where $D$ is required to be Hardy-type, thus 
relating our work to the work done in \cite{KuhlmannMatusinski} and \cite{KuhlmannMatusinski2}. 
Finally, in Section \ref{rankrealizedsection}, we show that for any pair of linearly ordered sets $(P,R)$ where $R$ is a principal final segment of $P$, there exists 
a field of generalized power series $K$ with a Hardy-type derivation making $K$ a H-field of principal differential rank $R$ and of principal unfolded differential rank $P$ (Theorem \ref{rankrealizedThm}).


 
 \section{Preliminaries}\label{preliminarysection}
    For us, a \textbf{partial map} on a set $A$ is a map from some subset $B$ of $A$ to $A$. 
    The domain of a partial map $\phi$ is denoted by $\dom\phi$. The identity map of a set will always be denoted by $id$.

    Every group considered in this paper is abelian. 
We recall that a \textbf{valued  group} is a group $G$ together with a map 
    $v:G\to\Gamma\cup\{\infty\}$, where $\Gamma\cup\{\infty\}$ is a totally ordered set with maximum 
    $\infty$ and such that for any $g,h\in G$, $v(g)=\infty\Leftrightarrow g=0$, $v(g)=v(-g)$ and 
    $v(g+h)\geq\min(v(g),v(h))$. $\Gamma$ is called the \textbf{value chain} of the valued group $(G,v)$. Note that for any $g,h\in G$, 
    $v(g-h)>v(g)$ implies $v(g)=v(h)$; this remark will be used a lot in this paper. 

If $(G,\leq)$ is an ordered abelian group, we use the notations $G^{\neq0}$ and $G^{<0}$ to respectively denote 
    $\{g\in G\mid g\neq0\}$ and $\{g\in G\mid g<0\}$. We say that two elements $g,h$ of the ordered group $(G,\leq)$ are \textbf{archimedean-equivalent} if 
   there are $n,m\in\N$ such that $\vert g\vert\leq n\vert h\vert$ and $\vert h\vert\leq m\vert g\vert$. We say that the ordered group $(G,\leq)$ is \textbf{archimedean} if any two non-zero elements of $G$ are 
   archimedean-equivalent. It is well-known that 
an ordered group is archimedean if and only if it is embeddable into $(\R,+,\leq)$, where $\leq$ denotes the usual order of $\R$ (see for example \cite{Fuchs}). 
We recall that, if $H$ is a convex subgroup of $(G,\leq)$, then $\leq$ naturally induces an order on the quotient group $G/H$ (see \cite{Fuchs}). We also recall that ordered abelian groups
can be seen as a particular examples of valued groups:
    if $(G,\leq)$ is an ordered abelian group, then we define its \textbf{natural valuation}, which we usually denote by $v_G$, as the valuation 
    given by the archimedean relation on $G$, i.e $v_G(g)\leq v_G(h)\Leftrightarrow \exists n\in\N, \vert h\vert\leq n\vert g\vert$.
 The set of non-trivial
   convex subgroups of $(G,\leq)$ is totally ordered by inclusion, its order-type is called the \textbf{rank} of the ordered group $(G,\leq)$. The rank of $(G,\leq)$ is 
   order-isomorphic to the set of final segments of $\Gamma$ via the map $H\mapsto v_G(H\backslash\{0\})$. 

    A \textbf{valued field}  is a a field $K$ endowed with a map $v:K \to G\cup\{\infty\}$ such that  $(K,+,v)$ is a valued group, $G$ is an ordered abelian group 
 and $v(xy)=v(x)+v(y)$ for all $x,y\neq0$.  $G$ is called the value group of $(K,v)$.
 We will implicitly consider ordered fields as particular instances of valued fields by considering the natural valuation associated to the given order.
    We use the following notation: $\valuedfield$ to mean that $(K,v)$ is a valued field with value group $G$, $v_G$ is the natural 
valuation of the ordered group $G$ and $\Gamma$ is the value chain of $(G,v_G)$. For a given valuation $v$ on a field $K$, we denote by 
$\ringv:=\{x\in K\mid v(x)\geq0\}$ its valuation ring, by $\idealv:=\{x\in K\mid v(x)>0\}$ the maximal ideal of $\ringv$ and by 
$\unitsv:=\{x\in K\mid v(x)=0\}$ the group of units of $\ringv$. We recall that a field valuation is entirely determined by its valuation ring; 
we will often assimilate a valuation with its associated ring. If $v$ and $w$ are two valuations on a field $K$, we say that 
$w$ is a \textbf{coarsening} of $v$  if $\ringv\subseteq\ringw$. We say that $w$ is a \textbf{strict coarsening} of $v$ if moreover $v\neq w$.
 The set of all strict coarsenings of a given valuation $v$ is totally ordered by inclusion, and its order type is 
what we call the \textbf{rank of the valued field $(K,v)$}.  We recall that the map 
  $\ringw\mapsto G_w:=v(\unitsw)$ defines a bijection from the set of strict coarsenings of $v$ to the set 
of non-trivial convex subgroups of $(G,\leq)$. Similarly, the map $G_w\mapsto\Gamma_w:=v_G(G_w^{\neq0})$ defines a bijection from the set of non-trivial convex
subgroups of $G$ to the set of final segments 
of $\Gamma$. It follows that 
 the rank of a valued field $\valuedfield$ has three equivalent characterizations: 
the first one is the order type of the set of all coarsenings of $v$; the second one is the order type of the set of all 
convex subgroups of $G$; finally, the third one is the order type of the set of all final segments of $\Gamma$.
We recall that if $w$ is a coarsening of $v$, then $v$ induces a valuation on the residue field $Kw$ which we denote by $\quotientval$. 
It is the valuation whose valuation ring is the image of $\ringv$ under the canonical homomorphism $\ringw\to Kw$. We also recall that, if $(K,\leq)$ is an ordered field and 
$w$ a coarsening of its natural valuation, then $\leq$ induces a field order on $Kw$ which we denote by $\leq_w$ 
and defined by $0+\idealw<_wa+\idealw\Leftrightarrow (0< a\wedge a\notin\idealw)$.

One way of constructing valued fields is to construct fields of generalized power series. 
Given a field $k$ and an ordered abelian group $G$, we define the field of power series $k((G))$ as the field 
$k((G)):=\{a=(a_g)_{g\in G}\in k^G\mid\text{ $\supp(a)$ is well-ordered}\}$  (where $\supp(a)$ denotes the support of $a$, i.e 
$\supp(a)=\{g\in G\mid a_g\neq0\}$)  endowed with component-wise addition and a multiplication 
defined as $(a_g)_{g\in G}(b_g)_{g\in G}=(c_g)_{g\in G}$ where $c_g=\sum_{h\in G}a_hb_{g-h}$. An element of $K$ is written as a formal sum $a:=\sum_{g\in G}a_gt^g$ with $a_g\in k$ for all $g\in G$. 
This field is naturally endowed with 
a valuation $v(a):=\min\supp(a)$. If $g=v(a)$, we say that $a_g$ is the \textbf{leading coefficient} and $a_gt^g$ the \textbf{leading term} of $a$.

 Finally, we want to recall some basic facts about quasi-orders, since Section \ref{phiranksection} makes use of them. 
A \textbf{quasi-order} (q.o) is a binary relation which is reflexive and transitive. All q.o's in this paper are assumed to be total, i.e they satisfy the condition 
$a\precsim b\vee b\precsim a$ for all $a,b$.
 A q.o $\precsim$ on a set $A$ naturally induces an equivalence
    relation which we denote by $\sim$ and defined as 
    $a\sim b\Leftrightarrow a\precsim b\precsim a$. The q.o $\precsim$ then naturally induces an order on the quotient 
$A/\sim$ defined as follows: we say that the $\sim$-class of $a$ is smaller than the $\sim$-class of $b$ if and only if 
$a\precsim b$ (one can easily check that this defines a total order on $A/\sim$). If $\precsim_1$ and $\precsim_2$ are two q.o's on $A$, we say that $\precsim_2$ is a \textbf{coarsening} of 
$\precsim_1$ if $a\precsim_1 b\Rightarrow a\precsim_2b$ for all $a,b\in A$. 
If $(A,\precsim)$ is a q.o set and $\phi:A\to A$ a map, we say that 
$\phi$ is \textbf{increasing} if $a\precsim b\Rightarrow\phi(a)\precsim\phi(b)$, and we say that $\phi$ is \textbf{strictly increasing} if $a\precsim b\Leftrightarrow\phi(a)\precsim\phi(b)$ for all $a,b\in A$.
     If  $B\subseteq A$, we say that $B$ is
\textbf{$\precsim$-convex} in $A$, or is convex in $(A,\precsim)$, if for any $a\in A$ and $b,c\in B$, $b\precsim a\precsim c$ implies 
$a\in B$. We say that $B$ is a a \textbf{final segment} of $(A,\precsim)$ if for any 
$a\in A$ and $b\in B$, $b\precsim a$ implies $a\in B$.  A final segment of $(A,\precsim)$ is said to be \textbf{principal} if it is of the form 
$B:=\{a\in A\mid b\precsim a\}$ for some $b\in A$. $B$ is then called the \textbf{principal final segment of $(A,\precsim)$ generated by $b$}. Note that orders and valuations are both particular instances of 
quasi-orders: indeed,  if $v$ is a valuation on a group $G$, then 
$v$ induces a quasi-order, which we denote by $\precsim_v$, defined by $g\precsim_v h\Leftrightarrow v(g)\geq v(h)$. 
\comment{It was proved in \cite{KuhlmannPointMatusinski} that a valuation $w$ on $K$ is a coarsening of $v$ if and only if the ring $\ringw$ is 
$\precsim_v$-convex. in particular, the set of coarsenings of $v$ can be identified with the set of 
$\precsim_v$-convex subrings of $K$; we implicitly use this fact in Section 1.
Finally, we recall that a quasi-ordered field is a field $K$ endowed with a q.o $\precsim$ which is 
either a field order or the quasi-order induced by a valuation.}

 \section{The $\phi$-rank of a valued field}\label{phiranksection}
 
 We now give a uniform approach to the notion of rank for a valued field with an operator, generalizing the already 
 existing notions of rank found in \cite{Kuhlmann} and \cite{KuhlmannPointMatusinski}. This will allow us to define the differential rank in Section \ref{diffranksubsec}. 
 The classical rank of a valued field $\valuedfield$ is characterized on three different levels: 
 at the level of the field $K$ itself, at the level of its value group $G$ and at the level of the valuation chain $\Gamma$. This is 
 why we now want to define three notions of $\phi$-ranks: one for quasi-ordered  sets, one for  groups and another one for 
 fields. Using quasi-orders instead of orders will be useful to give a certain characterization of the differential rank, see Theorem \ref{diffrankcharbyqosThm}.
 
 We first define the notion of rank of a quasi-ordered set. Let $(A,\precsim)$ be a q.o set. The \textbf{rank of the q.o set $(A,\precsim)$}
 is the order type of the set of all final segments of $(A,\precsim)$, ordered by inclusion. 
The \textbf{principal rank of the q.o set $(A,\precsim)$} is the order type of the set of principal final segments of $(A,\precsim)$. Note that the rank of 
$(A,\precsim)$ is the same as the rank of $(A/\sim,\leq)$, where $\leq$ is the order induced by $\precsim$ on $A/\sim$. Note also that the principal rank of $(A,\precsim)$ is isomorphic to 
 $(A/\sim,\leq^{\ast})$, where $\leq^{\ast}$ is the reverse order of $\leq$ (this is given by the order-reversing bijection $a\mapsto\{b\in A\mid a\precsim b\}$ from $A/\sim$ to 
 the set of principal final segments of $A$).
 \begin{Rem}\label{ordercompletionrem}
  A final segment is an increasing union of principal final segments. Therefore, the 
  rank is completely determined up to isomorphism by the principal rank (see \cite[Remark 2.13]{KuhlmannPointMatusinski}). Note that 
  the rank of a totally ordered set (as defined above) is in fact its order completion (see \cite[Definition 2.31]{Rosenstein}). \end{Rem}

Now let $\phi$ be a partial map on a set $A$. 
 We say that a subset $B$ of $A$ is \textbf{compatible with 
$\phi$}, or \textbf{$\phi$-compatible}, if for any $a\in A\cap\dom\phi$, $a\in B \Leftrightarrow\phi(a)\in B$. If $(A,\precsim)$ is a quasi-ordered set, 
$\phi$ a partial map on $A$ and $b\in A$, we say that 
$B\subseteq A$ is the \textbf{$\phi$-principal final segment of $(A,\precsim)$ generated by $b$} if  $B$ is the smallest $\phi$-compatible final 
segment of $(A,\precsim)$ containing $b$. We say that a final segment 
$B$ of $(A,\precsim)$ is \textbf{$\phi$-principal} if there is some $b\in A$ such that $B$ is $\phi$-principal generated by $b$.
\begin{Def}\label{phirankdef}
 \begin{enumerate}
  \item The  \textbf{$\phi$-rank} (respectively, the \textbf{principal $\phi$-rank}) of the quasi-ordered set $(A,\precsim)$ is the order type of the set of 
 $\phi$-compatible (respectively, $\phi$-principal) final segments of $(A,\precsim)$, ordered by inclusion.
 \item If $(G,\leq)$ is an ordered abelian group with a partial map $\phi$, we say that $H$ is  \textbf{the $\phi$-principal convex subgroup of $G$ generated by $g$} if $H$ is the smallest 
$\phi$-compatible convex subgroup of $G$ containing $g$.
 We  define 
the  \textbf{$\phi$-rank} (respectively, the \textbf{principal $\phi$-rank}) of the ordered group $(G,\leq)$ as the order type of the set of 
 $\leq$-convex $\phi$-compatible (respectively, $\phi$-principal) non-trivial subgroups of $G$.
 \item If $(K,v)$ is a valued field with a partial map $\phi$, we  say that a coarsening $w$ 
of $v$ is \textbf{$\phi$-compatible} if 
$\unitsw$ is $\phi$-compatible as a set. We say that $w$ is  \textbf{the $\phi$-principal coarsening of $v$ generated by $a$} 
if $\ringw$ is the smallest overring of $\ringv$ containing $a$ such that $w$ is $\phi$-compatible. We
define 
the  \textbf{$\phi$-rank} (respectively, the \textbf{principal $\phi$-rank}) of the valued field $(K,v)$ as the order type of the set of 
  $\phi$-compatible (respectively, $\phi$-principal) strict coarsenings of $v$. 
 \end{enumerate}

\end{Def}

\begin{Ex}\label{exemplesrang}
\begin{enumerate}[(a)]
 \item The rank of a q.o set $(A,\precsim)$ is equal to its $id$-rank.
 \item Let $\valuedfield$ be a valued field. Then the classical rank of $(K,v)$ as a valued field is 
 the $id$-rank of the valued field $(K,v)$. It is known that it is also equal to the $id$-rank of the ordered group $(G,\leq)$
  and to the $id$-rank of the ordered set $(\Gamma,\leq)$.
 \item Let $(K,\leq,\exp)$ be an ordered field endowed with   
  a $(GA),(T_1)$-logarithm as defined in \cite[Chapter 2, Section 1]{Kuhlmann}. Let $\phi$ be the logarithm $\exp^{-1}$ restricted 
 to $K^{>0}\backslash\ringv$, where $v$ is the natural valuation associated to $\leq$. One can easily check that our notion of compatibility coincides with the notion of compatibility defined in 
\cite[Chapter 3, Section 2]{Kuhlmann}. In particular, a coarsening $w$ of $v$ satisfies our definition of compatibility with $\phi$ if and only if 
 the logarithm is compatible with $w$ in the sense of \cite{Kuhlmann}. It follows that the
 exponential rank of $(K,\leq,\exp)$ is the
 $\phi$-rank of the valued field $(K,v)$.
 \item Let $(K,v,\sigma)$ be a valued difference field. Then one can check that a coarsening $w$ of $v$ satisfies our condition of 
 $\sigma$-compatibility if and only if $w$ is $\sigma$-compatible in the sense of \cite[Section 4]{KuhlmannPointMatusinski}. It follows that the 
difference rank of $(K,v,\sigma)$ defined in \cite[Section 4]{KuhlmannPointMatusinski} is equal to the $\sigma$-rank of the valued field $(K,v)$.
 \item The first author showed in \cite[Theorem 3.25]{Kuhlmann} that the exponential rank of 
 $(K,\leq,\exp)$ is also equal to the $\chi$-rank of $(G,\leq)$, where $\chi$ is the map 
 induced by the logarithm on the value group $G$ of $(K,v)$. She also showed that it is equal to the $\zeta$-rank of 
 $(\Gamma,\leq)$, where $\zeta$ is the map induced by $\chi$ on $\Gamma$. She together with Point and Matusinski also showed 
 similar results for difference fields in 
 \cite[Lemmas 4.5 and 4.6]{KuhlmannPointMatusinski}.
\end{enumerate}
\end{Ex}

Following Example \ref{exemplesrang} (e),
we now want to show that the $\phi$-rank of a valued field with an operator $\phi$ can be 
characterized at three different levels, as happens in the classical case. This can only be done if 
$\phi$ induces a map on its value group, which is why we need the following definition:
\begin{Def}\label{definducedmap}
We say that a map $\phi$ is \textbf{consistent} with a valuation $v$ if  $v(a)=v(b)\Rightarrow \phi(v(a))=\phi(v(b))$ for all $a,b\in\dom(\phi)$.
Let $\valuedfield$ be a valued field and $\phi$ a partial map on $K$. If $\phi$ is consistent with $v$, then 
$\phi$ naturally induces a partial map on $G$ defined by  
$\phi_G(v(a))=v(\phi(a))$. If $\phi_G$ is itself consistent with 
$v_G$, then it induces a partial map $\phi_{\Gamma}$ on $\Gamma$ defined by
 $\phi_{\Gamma}(v_G(g))=v_G(\phi_G(g))$.
\end{Def}

We then have the following result:

\begin{Prop}\label{fieldgroupchain}

 Let $\valuedfield$ be a valued field and $\phi$ a partial map on $K$. Assume that 
$\phi$ is consistent with $v$.
Then the $\phi$-rank (respectively, the principal $\phi$-rank ) of $(K,v)$ is equal to the $\phi_G$-rank (respectively, the principal $\phi_G$-rank) of 
$(G,\leq)$, where $\phi_G$ is the partial map of $G$ induced by $\phi$. If moreover 
$\phi_G$ is consistent with $v_G$, then 
the $\phi$-rank (respectively, the principal $\phi$-rank ) of $(K,v)$ is also equal to the $\phi_{\Gamma}$-rank (respectively, the principal $\phi_{\Gamma}$-rank) of 
$(\Gamma,\leq)$, where $\phi_{\Gamma}$ denotes the partial map of $\Gamma$ induced by $\phi_G$.

\end{Prop}
\begin{proof}
 We already know (see \cite[Lemma 3.4]{Kuhlmann}) that there is an inclusion-preserving bijection $\ringw\mapsto G_w:=v(\unitsw)$ between the set 
 of coarsenings of $v$ and the set of convex subgroups of $G$. 
 Let $\ringw$ be a coarsening of $v$  and $G_w$ the corresponding subgroup of $G$. Assume $\ringw$
 is $\phi$-compatible, let $0\neq g\in G$, $g=v(a)$. $g\in G_w\Leftrightarrow a\in \unitsw\Leftrightarrow\phi(a)\in \unitsw\Leftrightarrow \phi_G(g)=v(\phi(a))\in G_w$, 
so $G_w$ is $\phi_G$-compatible. Assume 
 $G_w$ is $\phi_G$-compatible. We have $a\in \unitsw\Leftrightarrow g=v(a)\in G_w\Leftrightarrow
 \phi_G(g)\in G_w\Leftrightarrow v(\phi(a))\in G_w\Leftrightarrow \phi(a)\in \unitsw$, so 
 $\ringw$ is $\phi$-compatible.  
 Therefore, there is an inclusion-preserving bijection between the set of $\phi$-compatible coarsenings of $v$ and 
 the set of $\phi_G$-compatible convex subgroups of $G$. 
 It also easy to see that a coarsening $\ringw$ is $\phi$-principal generated by $a$ if and only if $G_w$ is $\phi_G$-principal generated by $g:=v(a)$. Thus, we also have an 
 inclusion-preserving bijection between the set of $\phi$-principal coarsenings of $v$ and the set of $\phi_G$-principal convex subgroups of $G$. This proves the first claim of the proposition.
We could use an analogous proof to show that there is a bijection 
 between $\phi_G$-compatible convex subgroups of $G$ and $\phi_{\Gamma}$-compatible final segments of $\Gamma$ (using the map $G_w\mapsto v_G(G_w\backslash\{0\})$).
\end{proof}

\begin{Ex}
 \begin{itemize}
  \item Consider an ordered exponential field $(K,\leq,\exp)$ satisfying 
         axioms $(GA),(T_1)$ of \cite[Chapter 2, Section 1]{Kuhlmann} and denote by $v$ its natural valuation. Define $\phi$ as the logarithm 
$\log:=\exp^{-1}$ restricted to $K^{>0}\backslash\ringv$. One easily  sees that 
$\phi$ is consistent with $v$. The induced map on $G$ is the contraction map 
$\chi$ studied in \cite[Chapter 2, Section 7]{Kuhlmann}.  It then follows from Proposition \ref{fieldgroupchain} that the exponential rank 
of $K$ is equal to the $\chi$-rank of $G$ as stated in \cite[Theorem 3.25]{Kuhlmann}.
\item Consider a difference valued field $(K,v,\sigma)$, where $\sigma$ is compatible with $v$. One can check that $\sigma$ is consistent with $v$, 
so the difference rank of $K$ is equal to the $\sigma_G$-rank of  $G$.
 \end{itemize}

\end{Ex}

 In \cite[Section 5]{KuhlmannPointMatusinski}, the authors characterized the difference rank of a difference field in terms of an equivalence relation 
induced by $\sigma$. We now give similar results for the notion of $\phi$-rank in the case where $\phi$ is an increasing map. Assume then that 
$(A,\precsim)$ is a quasi-ordered set and that  
 $\phi$ is an increasing  map on $A$  with $\dom\phi=A$.

 We associate the following relations to $\phi$:
 
$a\qophi b\Leftrightarrow \exists n,k\in\N_0\quad \phi^n(a)\precsim\phi^k(b)$. 

$a\eqrel b\Leftrightarrow a\qophi b\wedge b\qophi a$.

\begin{Prop}\label{phiqoprop}
 The relation $\qophi$ is a quasi-order on $A$ and a coarsening of $\precsim$. Moreover, For any $\precsim$-convex subset $B$ of $A$ the following statements are equivalent: 

 \begin{enumerate}[(i)]
  \item $B$ is $\phi$-compatible 
  \item $B$ is $\qophi$-convex.
  \item  $B$ is $\eqrel$-closed, in the sense that for any $a,b\in A$ with $a\eqrel b$, $a\in B\Leftrightarrow b\in B$.
 \end{enumerate}
 In particular, the $\phi$-rank of $(A,\precsim)$ is equal to the rank of $(A,\qophi)$ and the principal $\phi$-rank of $(A,\precsim)$ is equal to the principal
rank of $(A,\qophi)$.

\end{Prop}

\begin{proof}
We start by showing the following: \\

  \begin{Claim*}
    If $a\qophi b$ and $b\qophi c$, then there are $j,k,n\in\N_0$ with $\phi^k(a)\precsim \phi^j(b)\precsim\phi^n(c)$. 
  \end{Claim*}
  \begin{proof}
   By definition of $\qophi$ there are $k,l,m,n\in\N_0$ with $\phi^k(a)\precsim\phi^l(b)$ and $\phi^m(b)\precsim\phi^n(c)$. Assume that 
$l\leq m$. Since $\phi$ is increasing, $\phi^k(a)\precsim\phi^l(b)$ implies \newline$\phi^{k+m-l}(a)\precsim\phi^m(b)\precsim\phi^n(c)$, hence 
the claim. If $m<l$, then $\phi^m(b)\precsim\phi^n(c)$ implies $\phi^k(a)\precsim \phi^l(b)\precsim\phi^{n+l-m}(c)$.
\end{proof}

 Now we prove the proposition. Obviously, $\qophi$ is reflexive and total because $\precsim$ is, and thanks to the claim $\qophi$ is also transitive, so $\qophi$ is a quasi-order. Note that $\eqrel$ is the equivalence relation 
induced by the q.o $\qophi$.
Now let $B$ be $\precsim$-convex. By definition of convexity, (ii) implies (iii). 
Since $a\eqrel \phi(a)$ is true for all $a$, (iii) implies (i). Now let us prove that (i) implies (ii), so assume $B$ is $\phi$-compatible. Let $a,c\in B$ and $b\in A$ such that $a\qophi b\qophi c$. By the claim, there are 
$j,k,n\in\N_0$ with $\phi^k(a)\precsim \phi^j(b)\precsim\phi^n(c)$. Since $B$ is $\phi$-compatible we have $\phi^k(a),\phi^n(c)\in B$. Since $B$ is 
$\precsim$-convex, $\phi^k(a)\precsim \phi^j(b)\precsim\phi^n(c)$ implies $\phi^j(b)\in B$. By $\phi$-compatibility, this implies $b\in B$. 
 This shows that $B$ is $\qophi$-convex. 

Now let us prove the last statement. We show that a subset $B$ of $A$ is a  final segment of $(A,\qophi)$ if and only if it is a $\phi$-compatible 
 final segment of $(A,\precsim)$. Let $B$ be a $\phi$-compatible final segment of 
  $(A,\precsim)$. Let $b\in B$ and $a\in A$ with $b\qophi a$. If $b\eqrel a$ then since (i) implies (iii) we have $a\in B$. If $b\precnsim_{\phi}a$, then since $\qophi$ is a coarsening of $\precsim$ we must have 
$b\precnsim a$. Since $B$ is a final segment of $(A,\precsim)$, it follows that $a\in B$. This shows that $B$ is a final segment of $(A,\qophi)$. 
 Conversely, assume that $B$ is a final segment of 
$(A,\qophi)$. Since $\qophi$ is a coarsening of $\precsim$ then $B$ must also be a final segment of $(A,\precsim)$. 
 In particular, $B$ is $\precsim$-convex and 
$\qophi$-convex. Since (ii) implies (i), 
$B$ must be $\phi$-compatible. 
This shows that the set of $\phi$-compatible final segments of $(A,\precsim)$ coincides with the set of 
final segments of $(A,\qophi)$. It then immediately follows that the 
set of $\phi$-principal final segments of $(A,\precsim)$ coincides with the set 
of principal final segments of $(A,\qophi)$.
\end{proof}

\begin{Rem}\label{qoremark}
 By applying Propositions \ref{fieldgroupchain} and \ref{phiqoprop} to the case of exponential fields and to the case of difference fields, we recover 
 \cite[Theorem 3.30 and Corollary 3.34]{Kuhlmann} and \cite[Corollaries 5.4 and 5.5]{KuhlmannPointMatusinski}. Our relation $\eqrel$ corresponds to 
 $\sim_{\sigma_{\Gamma}}$ in \cite[Section 5]{KuhlmannPointMatusinski} 
and to $\sim_{\zeta}$ in \cite[Chapter 3, Section 4]{Kuhlmann}.
\end{Rem}

\section{Asymptotic Couples}\label{acsubsection}

  This section is dedicated to the study of asymptotic couples. We introduce a notion of cut point which allows us to describe the structure 
  of an asymptotic couple. 
  Asymptotic couples naturally arise as the value group of differential-valued fields, which is why 
  the results of this section will be useful in Section \ref{differentialranksection} to study the differential rank of a differential-valued field.
  
  We recall that an \textbf{asymptotic couple} is a pair 
$(G,\psi)$ consisting of an ordered abelian group $G$ and a map $\psi: G^{\neq0}\to G$ satisfying the following conditions: 
 \begin{enumerate}
   \item[(AC1)] $\psi(g+h)\geq\min(\psi(g),\psi(h))$ for any $g,h$.
   \item[(AC2)] $\psi(ng)=\psi(g)$ for any $g\in G$, $n\in\N$.
   \item[(AC3)] $\psi(g)<\psi(h)+\vert h\vert$ for any $g,h\in G^{\neq0}$.
  \end{enumerate}
  We say that $(G,\psi)$ is a \textbf{H-type asymptotic couple} if moreover the following condition is satisfied: 
  \begin{enumerate}
   \item[(ACH)]$\forall g,h\neq0, g\leq h<0\Rightarrow \psi(g)\leq\psi(h).$
  \end{enumerate}

 H-type asymptotic couples have the good property that the map $\psi$ is constant on archimedean classes of $G$, which is not the case for 
  asymptotic couples in general. This will become important in Section \ref{powerseries} when we define a derivation on power series.
  Note that if $(G,\psi)$ is an asymptotic couple, then $(G,\psi')$ is still an asymptotic couple for any translate 
  $\psi':=\psi+h$ of $\psi$ (where $h\in G$). 
  \begin{Notation}
  In all the rest of this section, $(G,\psi)$ is an asymptotic couple. 
We will denote by $\leq$ the order of $G$,  by $v_G$  the natural valuation associated to $(G,\leq)$ and by 
   $\sim$ the equivalence relation defined by $g\sim h\Leftrightarrow v_G(g)=v_G(h)$.  If $g\in G$ and $H\subseteq G$, we write $g<H$ to mean that 
  $g<h$ for all $h\in H$; note that if $H$ is a convex subgroup this implies $v_G(h)>v_G(g)$ for all $h\in H$. If $(G,\psi)$ happens
  to be H-type, then $\psi$ is consistent with $v_G$, in which case we will denote by $\omega$ the map induced by $\psi$ on $\Gamma:=v_G(G^{\neq0})$. 
  We denote by $D_G$ the map $D_G(g):=\psi(g)+g$ from $G^{\neq0}$ to $G$. We denote by $\Psi$ the set $\psi(G^{\neq0})$.
 We know from \cite[Proposition 2.3(4)]{Aschdries2}  that $D_G:G^{\neq0}\to G$ is strictly increasing, so in particular it is injective.
  The asymptotic couple $(G,\psi)$ is said to have \textbf{asymptotic integration} if the map 
$D_G$ is surjective. 
 \end{Notation}

  We now want to describe the behavior of the map $\psi$. In order to do this, we introduce the following notion: 
  \begin{Def}\label{cutpointdef}
   We say that $c\in G$ is a \textbf{cut point} for $\psi$ if 
$v_G(\psi(g))>v_G(g)\Leftrightarrow v_G(c)>v_G(g)$ for any $g\in G^{\neq0}$. A cut point $c$ for $\psi$ is called \textbf{regular} if $c=0$ or if $c\neq0$ has the same sign 
as $\psi(c)$.
  \end{Def}

  We want to show that cut points always exist, and we want to show how cut points can be used to describe the behavior of $\psi$. 
  To do this, we will need the following two lemmas:
 
\begin{Lem}\label{Fundlemofascouples}
 Let $(G,\psi)$ be an asymptotic couple and $g,h\in G^{\neq0}$ with $g\neq h$. Then 
$v_G(\psi(g)-\psi(h))>v_G(g-h)$.
\end{Lem}
\begin{proof}
 See \cite[Proposition 2.3.]{Aschdries2}.
\end{proof}

   \begin{Lem}\label{lemmacutpoint}
    Let $(G,\psi)$ be an asymptotic couple. Then the following holds:
      \begin{enumerate}[(i)]
	\item For any $g\in G^{\neq0}$ with $v_G(g)\geq v_G(\psi(g))$, if $v_G(h)\geq v_G(\psi(g))$ then $v_G(\psi(h)-\psi(g))>v_G(\psi(g))$;  in particular 
	$\psi(h)$ and $\psi(g)$ are archimedean-equivalent and have the same sign.
	\item For any $c\in G^{\neq0}$, $c\in G$ is a cut point for $\psi$ if and only if $\psi(c)\sim c$.
	\item For any $c,g\in G$, if $c$ is a cut point for $\psi$, then $g$ is a cut point for $\psi$ if and only if $g\sim c$.
       \item For any $g\in G^{\neq0}$, if $v_G(g)\geq v_G(\psi(g))$, then $\psi(g)$ is a regular cut point for $\psi$.
       \item If $(G,\psi)$ is H-type and $c\in G^{\neq0}$ is a cut point for $\psi$, then 
       $\psi(c)$ is a fixpoint of $\psi$.
      \end{enumerate}    
   \end{Lem}

   \begin{proof}
    For (i): By Lemma \ref{Fundlemofascouples}, $v_G(\psi(h)-\psi(g))>v_G(h-g)\geq v_G(\psi(g))$.
    For (ii): Assume $c$ is a cut point, then by definition we have $v_G(c)\geq v_G(\psi(c))$, hence by (i) $\psi(\psi(c))\sim \psi(c)$. If $v_G(c)>v_G(\psi(c))$
    were true, we would have by definition of cut point 
    $v_G(\psi(\psi(c)))> v_G(\psi(c))$, which is impossible; thus, we must have $\psi(c)\sim c$. Conversely, assume 
   $\psi(c)\sim c$ and let $h\in G$ with $v_G(c)>v_G(h)$. If $v_G(h)\geq v_G(\psi(h))$ were true, we would have by (i) $\psi(c)\sim\psi(h)$, a contradiction, so 
   $v_G(\psi(h))>v_G(h)$ must hold. Now take $v_G(h)\geq v_G(c)$; then $v_G(h)\geq v_G(\psi(c))$, which by (i) implies we have $\psi(h)\sim\psi(c)$ hence $v_G(h)\geq v_G(\psi(h))$. 
   This proves that $c$ is a cut point.
     (iii) follows directly from the definition of cut point.
     For (iv): Let $0\neq g$ with $v_G(g)\geq v_G(\psi(g))$. Then by (i) $\psi(\psi(g))\sim\psi(g)$, which by (ii) means that $\psi(g)$ is a cut point. Moreover,
     $\psi(g)$ has the same sign as $\psi(\psi(g))$ by (i).
    For (v):   
    Since $(G,\psi)$ is H-type, $\psi$ is constant on archimedean classes of $G$, so $c\sim\psi(c)$ implies $\psi(c)=\psi(\psi(c))$.       
   \end{proof}

    We can now describe the behavior of $\psi$ and $D_G$ using a cut point (we recall that $D_G=\psi+id$). Roughly speaking, the role of a cut point $c$ is to separate the group 
    into two parts, on each of which $\psi$ has a different behavior. Indeed, $\psi$ acts like a centripetal precontraction (see Definition \ref{contractiondef}) on 
    the set of $g$'s with $v_G(g)<v_G(c)$, whereas $\psi(g)$ is close to $\psi(c)$ for all $g\in G$ with $v_G(g)\geq v_G(c)$. 
    This is given by the next proposition, which will have important consequences for the differential rank in Section \ref{differentialranksection}:
   \begin{Prop}\label{behaviourofDG}
    Let $(G,\psi)$ be an asymptotic couple. Then $G$ admits a regular cut point for $\psi$. Moreover, for any regular cut point $c$ and any $g\neq0$, the following holds: 
	   \begin{enumerate}[(i)]
   \item If $v_G(g)\geq v_G(c)\neq\infty$, then $\psi(g)\sim c$, $v_G(\psi(g)-\psi(c))\geq v_G(c)$ and $\psi(g)$ has the same sign as $c$. 
	 \item If  $v_G(g)> v_G(c)\neq\infty$, then $D_G(g)\sim c$ and  $D_G(g)$ has the same sign as $c$.
    \item If $0\neq g\sim c$, then $v_G(D_G(g))\geq v_G(c)$.
    \item If $v_G(c)> v_G(g)$, then $v_G(\psi(g))>v_G(g)$, $D_G(g)\sim g$ and $D_G(g)$ has the same sign as $g$.
  \end{enumerate}
   \end{Prop}

    \begin{proof}
     Assume $0$ is not a cut point. This means there is $g\in G$ with $v_G(g)\geq v_G(\psi(g))$, which by Lemma \ref{lemmacutpoint}(iv) means that 
      $\psi(g)$ is a regular cut point for $\psi$. Now let $c$ be a regular cut point for $\psi$ and $0\neq g$. (i)
      follows from Lemma \ref{lemmacutpoint}(i) and (ii).  For (ii): since 
      by (i) $v_G(g)>v_G(\psi(g))$, it follows that $g+\psi(g)$ is equivalent to and has the same sign as $\psi(g)$, and we conclude by (i).
 For (iii): $v_G(D_G(g))=v_G(\psi(g)+g)\geq v_G(g)$ (because $g\sim\psi(g)$). For (iv): $v_G(\psi(g))\geq  v_G(g)$ follows from the definition of cut point.
 It then follows that $\psi(g)+g$ is equivalent to and has the same sign as $g$.
    \end{proof}
    
 Proposition \ref{behaviourofDG} allows us to give a criterion to decide whether a given convex subgroup is $\psi$-compatible. This will be given by the next three lemmas:
    \begin{Lem}\label{positivepsicutpointLem}
     	 If $c\neq0$ is a cut point and $g\in G^{\neq0}$ is such that $v_G(c)> v_G(\psi(g))$, then $\psi(g)$ is negative. In particular, $\psi(g)>0$ implies that 
      $\psi^2(g)$ is a non-zero cut point for $\psi$.
    \end{Lem}
\begin{proof}
 By Lemma \ref{lemmacutpoint}(ii), we have $v_G(\psi(c)+\vert c\vert)\geq v_G(c)$. If $\psi(g)>0$, then $v_G(c)>v_G(\psi(g))$ implies 
    $\psi(c)+\vert c\vert<\psi(g)$ which contradicts (AC3). Thus, $\psi(g)>0$ implies $v_G(\psi(g))\geq v_G(c)$, 
    which by Proposition \ref{behaviourofDG}(i) implies $\psi^2(g)\sim c$. Since $c\neq0$, it follows that $\psi^2(g)\neq0$. 
    By Lemma \ref{lemmacutpoint}(iii), $\psi^2(g)$ is then a non-zero cut point for $\psi$.
\end{proof}

 \begin{Lem}\label{cinHlemma}
  Let $H$ be a convex subgroup of $G$, $c$ a cut point for $\psi$ and assume $c\in H$. Then $g\in H\Rightarrow \psi(g)\in H$.
 \end{Lem}
  \begin{proof}
   Assume $g\in H$. If $v_G(g)\geq v_G(c)$, then Proposition \ref{behaviourofDG}(i) implies 
   $\psi(g)\sim c$. Since $c\in H$, it follows from the convexity of $H$ that 
   $\psi(g)\in H$. If $v_G(c)> v_G(g)$, then it follows from Proposition \ref{behaviourofDG}(iv) that
   $v_G(\psi(g))> v_G(g)$. Since $g\in H$, it follows from the convexity of $H$ that 
   $\psi(g)\in H$.
  \end{proof}

\begin{Lem}\label{compatiblecontainsf}
 Let $c$ be a cut point for $\psi$ and $H$  a non-trivial convex subgroup of $G$. Then $H$  is compatible with $\psi$ if and only if the following two conditions hold: 
\begin{enumerate}[(i)]
 \item $c\in H$.
\item For any $g\in G$ with $v_G(c)> v_G(g)$, $\psi(g)\in H\Rightarrow g\in H$.
\end{enumerate}

 Moreover, if $H$ is compatible with $\psi$, then $\psi(g)<H$ for any $g\in G\backslash H$.
\end{Lem}

 \begin{proof}
  Assume that $H$ is compatible with $\psi$.
  Then clearly (ii) is true. Towards a contradiction, 
  assume that $c\notin H$ and take $g\in H$. By convexity of $H$, we have $v_G(g)> v_G(c)$. By Proposition \ref{behaviourofDG}(i), this 
  implies $\psi(g)\sim c$.  It follows from the convexity of $H$ that $\psi(g)\notin H$, which contradicts the fact that 
$H$ is $\psi$-compatible. This proves that (i) must hold. Conversely, assume (i) and (ii) and let us prove that 
$H$ is compatible with $\psi$. It follows from condition
(i) and from Lemma \ref{cinHlemma} that $g\in H\Rightarrow\psi(g)\in H$. Now assume that 
$\psi(g)\in H$ and let us show $g\in H$. If $v_G(c)> v_G(g)$, this follows from condition (ii).
If $v_G(g)\geq v_G(c)$, then since $c\in H$ it follows from convexity of $H$ that $g\in H$. This proves that 
$H$ is compatible with $\psi$.
Now  let $g\in G\backslash H$. Then $\psi(g)\notin H$,
 and since $c\in H$, the convexity of $H$ implies $v_G(c)> v_G(\psi(g))$. It follows from Lemma \ref{positivepsicutpointLem} that $\psi(g)$ is negative. By convexity of $H$, we then 
 have $\psi(g)<H$.
 \end{proof}

Proposition \ref{behaviourofDG} shows that the behavior of $\psi$ is particularly simple if $0$ is the cut point for $\psi$ 
(note that if $0$ is a cut point for $\psi$ then it is the only cut point). Indeed, in that case, 
$\psi$ acts like a centripetal precontraction map on $G^{<0}$. Therefore, it can be practical to transform a given map $\psi$ into another 
map $\psi'$ which has $0$ as a cut point. One way of doing that is to translate $\psi$ by a gap or by the maximum of $\Psi=\psi(G^{\neq0})$ if it exists. We recall that
 a 
\textbf{gap} for $\psi$ is an element $g\in G$ such that $\Psi<g<D_G(G^{>0})$. Aschenbrenner and v.d.Dries showed that the existence of a gap or of a maximum of $\Psi$ is 
connected to the existence of asymptotic integration in $(G,\psi)$.
 We recall the following results (see \cite[Proposition 2.2 and Theorem 2.6]{Aschdries2}): 

\begin{Lem}
 Let $(G,\psi)$ be an asymptotic couple. The following holds:
 \begin{enumerate}[(i)]
  \item $\psi$ has at most one gap.
    \item $G\backslash D_G(G^{\neq0})$ has at most one element.
  \item If $g$ is a gap for $\psi$ or the maximum of $\Psi$ then $G\backslash D_G(G^{\neq0})=\{g\}$. In particular, 
  if $(G,\psi)$ has asymptotic integration, then $\Psi$ has no maximum and $\psi$ has no gap.
 \end{enumerate}

\end{Lem}

Gaps and maximum of $\Psi$ are connected to our notion of cut point; more precisely, we have the following:
 \begin{Lem}\label{gapiscutpointLem}
  \begin{enumerate}[(i)]
   \item If $g$ is a gap for $\psi$ or the maximum of $\Psi$, then $g$ is a regular cut point for $\psi$.
   \item If $0$ is a cut point for $\psi$, then $0=\sup\Psi$ and $(G,\psi)$ does not have asymptotic integration.
  \end{enumerate}

 \end{Lem}

\begin{proof}
For (ii):  Since $0$ is the cut point, Proposition \ref{behaviourofDG}(iv) implies that $D_G(h)\sim h$  for all $h\in G^{\neq0}$. This implies 
in particular that $0\notin D_G(G^{\neq0})$, which proves that $(G,\psi)$ does not have asymptotic integration.
Assume there is $g$ with 
$\psi(g)>0$. By (AC3), we  have for all $h\neq0$: $0<\psi(g)<D_G(\vert h\vert)$, which implies $v_G(\psi(g))\geq v_G(D_G(\vert h\vert))$.
Since $D_G(\vert h\vert)\sim h$, we thus have
 $v_G(\psi(g))\geq \vert h\vert$.
 It follows that $\psi(\psi(g))=0$ (otherwise, taking $h:=\psi(\psi(g))$ we would have $v_G(\psi(g))\geq v_G(\psi(\psi(g)))$, which contradicts the fact the $0$ is the cut point).
 But then by (AC3), 
$\psi(g)<\psi(\psi(g))+\psi(g)=\psi(g)$, a contradiction. Thus, $\sup\Psi\leq0$. For any $g<0$, since $0$ is a cut point we have 
$g<\psi(g)<0$, hence $\sup\Psi=0$.
 For (i): Let $c$ be a regular cut point for $\psi$. If $c=0$, then by (ii) we must have $g=0=c$. Assume $c\neq0$. 
 Assume first that $v_G(c)\neq\max\Gamma$. Then there is 
 $h\in G^{\neq0}$ such that $v_G(h)>v_G(c)$. By Proposition \ref{behaviourofDG}(ii), 
 $D_G(\vert h\vert)$ has the same sign as $c$ and $D_G(\vert h\vert)\sim c$.
 By assumption on $g$, we have $\psi(c)\leq g< D_G(\vert h\vert)$. It follows that 
 $g\sim c$ and that $g$ has the same sign as $c$. By Proposition 
 \ref{behaviourofDG}(i), $\psi(g)$ has the same sign as $c$ and $\psi(g)\sim c$. 
 It follows that $\psi(g)$ as the same sign as $g$ and $\psi(g)\sim g$, so 
 $g$ is a regular cut point by Lemma \ref{lemmacutpoint}(ii). Now assume that 
 $v_G(c)=\max(\Gamma)$. It follows from Proposition \ref{behaviourofDG}(i) that 
 $\psi(h)=\psi(c)$ for any $h$ with $h\sim c$. Take $h$ such that 
 $0<h\leq \vert\psi(c)\vert$. Then $h\sim c$, and 
 $D_G(h)=h+\psi(h)=h+\psi(c)$ has the same sign as $\psi(c)$.
  By assumption on $g$,  we have 
 $\psi(c)\leq g< D_G(h)$. It follows that $g\sim c$, hence $\psi(g)=\psi(c)$. In particular, $g$ has the same sign as 
 $\psi(g)$ and $g\sim\psi(g)$. This shows that $g$ is a regular cut point. 
\end{proof}

As we mentioned above, if a gap or a maximum of $\Psi$ exists, then  we can transform $\psi$ into a map which has $0$ as a cut point. This will become important in Section \ref{unfrksection}:
 \begin{Lem}\label{translatebygap}
  Assume that $c\in G$ is either a gap for $\psi$ or the maximum of $\Psi$. Define 
$\psi'(g):=\psi(g)-c$. Then $0$ is a cut point for $\psi'$.
 \end{Lem}

 \begin{proof}
  Just note that $0$ is a gap for $\psi'$ or a maximum of $\Psi'$, so the claim follows from 
Lemma \ref{gapiscutpointLem}(i).
 \end{proof}

   We now want to focus on the case where $(G,\psi)$ is H-type. In that case, there is a canonical choice for a regular cut point, namely the fixpoint of $\psi$:
\begin{Lem}\label{psinegativeLem}
 Assume $(G,\psi)$ is H-type. Then $0$ is a cut point for $\psi$ if and only if $\psi$ has no fixpoint. If $\psi$ has a fixpoint,
 then it is unique and it is a regular cut point for $\psi$. Moreover, if $\psi$ has a negative fixpoint, then 
 $\psi$ only takes negative values.
\end{Lem}
\begin{proof}
 If $0$ is a cut point then by definition $v_G(\psi(g))> v_G(g)$ holds for every $g\neq0$ so $\psi$ has no fixpoint. If $c$ is a cut point with $c\neq0$, then by 
Lemma \ref{lemmacutpoint}(v) $\psi(c)$ is a fixpoint of $\psi$. If $d$ is another fixpoint then by Lemma \ref{lemmacutpoint}(ii) it must be a cut point, so by Lemma 
\ref{lemmacutpoint}(iii) $d\sim\psi(c)$, but since $(G,\psi)$ is H-type this implies $\psi(\psi(c))=\psi(d)$ hence $\psi(c)=d$. This proves the 
uniqueness of the fixpoint.   If $c$ is the fixpoint of $G$, then we have in particular $c\sim\psi(c)$ which by Lemma \ref{lemmacutpoint}(ii) implies that $c$ is a cut point. 
Finally, assume $c<0$. Since $c=\psi(c)$ is a cut point, we have $\psi(g)<0$ for any $v_G(g)\geq v_G(c)$ by Proposition \ref{behaviourofDG}(i). Now take $g$ with 
$v_G(c)>v_G(g)$. Since $(G,\psi)$ is H-type, we have 
$\psi(g)\leq \psi(c)=c<0$. Thus, $\psi(g)$ is negative for every $g$.
\end{proof}

  Proposition \ref{behaviourofDG}(i) and (iv) allows us to describe the behavior of the map $\omega$:

\begin{Lem}\label{behaviouromegaLem}
 Let $(G,\psi)$ be a H-type asymptotic couple and 
denote by $\omega$ the map induced by $\psi$ on $\Gamma$. Let $c$ be a cut point for $\psi$ and 
$\alpha:=v_G(c)$. Then we have $\omega(\gamma)>\gamma$ for all $\gamma$ with $\gamma<\alpha$ and 
$\omega(\gamma)=\alpha$ for all $\gamma\in\Gamma$ with $\alpha\leq\gamma$.

\end{Lem}

  In the H-type case, the existence of a gap is part of  a trichotomy (see \cite[Lemma 2.4]{Gehret}): 
  \begin{Prop}\label{trichotomy}
   Let $(G,\psi)$ be a H-type asymptotic couple. Then exactly one of the following holds:
    \begin{enumerate}[(i)]
     \item $\psi$ has a gap
      \item $\Psi$ has a maximum
      \item $(G,\psi)$ admits asymptotic integration
    \end{enumerate}

  \end{Prop}

 Sometimes, it is practical to work with an asymptotic couple where $\psi$ only takes negative values. The next lemma will be useful in that 
regard, especially for the proof of Theorem \ref{diffrankcharbyqosThm}:

\begin{Lem}\label{translatetonegativeLem}
  Let $(G,\psi)$ be a H-type asymptotic couple. There exists $x\in G^{<0}$ such that, if $\hat{\psi}$ denotes the map 
  $\hat{\psi}:=\psi+x$, $\alpha:=v_G(x)$ and $\hat{\omega}$ is the map induced by $\hat{\psi}$ on $\Gamma$, the following holds: 
   \begin{enumerate}[(i)]
    \item The $\psi$-rank (respectively, the principal $\psi$-rank) of $(G,\leq)$ is equal to the 
$\hat{\psi}$-rank (respectively, the principal $\hat{\psi}$-rank) of $(G,\leq)$.
    \item $\hat{\psi}(g)<0$ for all $g\in G^{\neq0}$.
    \item for all $\gamma\in\Gamma$, $\hat{\omega}(\gamma)=\min(\alpha,\omega(\gamma))$.
    \item For any $\gamma,\delta\in\Gamma$, $\alpha\leq\gamma\wedge\alpha\leq\delta\Rightarrow\alpha\leq\omega(\gamma)=\omega(\delta)$.
     
   \end{enumerate}
 
 \end{Lem}
\begin{proof}
 If $\psi(g)<0$ for all $g\neq0$ then  we can take $x=0$, so assume that there is $g$ with $\psi(g)\geq0$.
 We distinguish two cases:  
  \begin{enumerate}[1.]
   \item $0=\max\Psi$. In that case set $c=0$ and choose $x<0$ with $\psi(x)=0$.
    \item there exists $g$ with $\psi(g)>0$. In that case define $c$ as the fixpoint of $\psi$ which exists and is positive by Lemma \ref{psinegativeLem};
set $x:=-2c$.
  \end{enumerate}

Note that in both cases, $c$ is a cut-point for $\psi$, $x$ is a cut point for $\hat{\psi}$ (because $\hat{\psi}(x)\sim x$), $\psi(x)=c$ and $v_G(c)\geq v_G(x)$. 
Let us show (i). We  just have to show 
that for any convex subgroup $H$ of $G$, $H$ is $\psi$-compatible if and only if it is $\hat{\psi}$-compatible.
We use Lemma \ref{compatiblecontainsf}. Assume then that $H$ is $\psi$-compatible. Then by Lemma \ref{compatiblecontainsf} $c\in H$. Since $\psi(x)=c$, it follows by 
$\psi$-compatibility of $H$ that $x\in H$. Now take $g$ with $\hat{\psi}(g)\in H$.
We have $\psi(g)=\hat{\psi}(g)-x$, hence $\psi(g)\in H$. 
By $\psi$-compatibility, this
implies $g\in H$. This proves by Lemma \ref{compatiblecontainsf} that $H$ is $\hat{\psi}$-compatible. Conversely, assume that $H$ is 
$\hat{\psi}$-compatible. Then $x\in H$, and since $v_G(c)\geq v_G(x)$ it follows by convexity that $c\in H$. Now take $g\in G$ with 
$\psi(g)\in H$. Then $\hat{\psi}(g)=\psi(g)+x\in H$, which by $\hat{\psi}$-compatibility implies $g\in H$. It follows by Lemma \ref{compatiblecontainsf} that $H$ is 
$\psi$-compatible.

Now let us show (ii),(iii),(iv). We first consider case  1. In that case, $x$ is the fixpoint of $\hat{\psi}$, so (ii) follows from Lemma \ref{psinegativeLem}. Clearly, by definition of $\hat{\psi}$, 
$v_G(x)> v_G(\psi(g))\Rightarrow \psi(g)\sim\hat{\psi}(g)$. If $v_G(\psi(g))>v_G(x)$, then by definition of $\hat{\psi}$ we have 
$\hat{\psi}(g)\sim x$. Since $\Psi<0$, $\psi(g)$ and $x$ have the same sign, so $\psi(g)\sim x$ also implies $\hat{\psi}(g)\sim x$. 
This proves (iii). If $v_G(g)\geq v_G(x)$, (ACH) implies $\psi(g)\geq\psi(x)=0$, and since $0=\max\Psi$ it follows that $\psi(g)=0$, hence (iv). Now let us consider case 
2. By (AC3), we have $\psi(g)<c+\psi(c)=2c$ for all $g\in G^{\neq0}$, hence $\hat{\psi}(g)<0$, hence (ii). Clearly, by definition of 
$\hat{\psi}$, $v_G(x)>v_G(\psi(g))\Rightarrow \psi(g)\sim\hat{\psi}(g)$. Now assume  
$v_G(\psi(g))\geq v_G(x)$.  By Proposition \ref{behaviourofDG}(i), 
$v_G(\psi(g)-c)>v_G(c)$, hence $\hat{\psi}(g)=\psi(g)-2c\sim c\sim x$. This proves (iii). Now take $g,h$ with $v_G(g)\geq v_G(x)$ and $v_G(h)\geq v_G(x)$. Since $x\sim c$, 
Proposition \ref{behaviourofDG}(i) implies $\psi(g)\sim c\sim x\sim \psi(h)$, hence (iv).
\end{proof}

\section{The differential rank}\label{differentialranksection}
 
   This section introduces and studies the differential rank. We first want to recall the definitions of the 
   different classes of fields which we are interested in.

 We recall that, following 
   Rosenlicht's definition (see \cite{Rosenlicht}), a \textbf{differential-valued field} is a triple $(K,v,D)$, where 
   $v$ is a field valuation on $K$ and $D$ a derivation such that the following is satisfied:
   \begin{enumerate}
    \item[(DV1)] $\ringv=\idealv+\mathcal{C}$, where $\mathcal{C}$ is the field of constants of $(K,D)$.
    \item[(DV2)] If $a\in\ringv,b\in\idealv$ and $b\neq0$, then $v(D(a))>v(\frac{D(b)}{b})$.
   \end{enumerate}

   If $(K,v,D)$ only satisfies (DV2), then we say that it is a \textbf{pre-differential-valued field}. Note that (DV2) implies that $v$ is trivial on $\constants$.
   If $(K,v,D)$ is a pre-differential valued field, we will denote by $\phi$ its logarithmic derivative restricted to elements of non-trivial valuation, i.e 
   $\phi: K\backslash(\unitsv\cup\{0\})\to K$, $a\mapsto\frac{D(a)}{a}$. Note that if $(K,v,D)$ is a pre-differential-valued field and $a\in K$, $a\neq0$, then 
$(K,v,aD)$ is also a pre-differential-valued field, where $aD$ denotes the derivation $b\mapsto aD(b)$.
   The notion of pre-differential-valued field was 
   introduced by Aschenbrenner and v.d.Dries  in \cite{Aschdries2}, where 
they showed in particular that any pre-differential-valued field can be embedded into a differential-valued field. 
A pre-differential-valued field
$(K,v,D)$ is said to have \textbf{asymptotic integration} if for any $a\in K$, there exists 
$b\in K$ such that $v(D(b)-a)>v(a)$. F.-V. Kuhlmann showed in \cite{FVKuhlmanndifffields} that if $K$ is spherically complete and has asymptotic integration, then it has integration.

 In \cite{Aschdries2}, the authors also introduced a class of  differential-valued fields called H-fields which turn out to be 
particularly significant for the theory of transseries and the model-theoretic study of Hardy fields. They also introduced the weaker notion of pre-H-field and showed that any pre-H-field can be embedded 
into a H-field. We recall their definition: 
A \textbf{pre-H-field} is a valued ordered differential field $(K,v,\leq,D)$ such that :
\begin{enumerate}
 \item[(PH1)] $(K,v,D)$ is a pre-differential-valued field
 \item[(PH2)] $\ringv$ is $\leq$-convex.
 \item[(PH3)] for all $a\in K,a>\ringv\Rightarrow D(a)>0$.
\end{enumerate}

\begin{Def}\label{Hfielddef}
 An \textbf{H-field} is an ordered differential field $(K,\leq,D)$ such that 
    \begin{enumerate}
     \item[(H1)] $(K,v,\leq,D)$ is a pre-H-field, where $\ringv:=\{a\in K\mid\exists c\in\constants, \vert a\vert\leq c\}$.
     \item[(H2)] $\ringv=\constants+\idealv$
    \end{enumerate}
\end{Def}

    If $(K,v,D)$ is a pre-differential-valued field, then axiom (DV2) implies that the logarithmic derivative $\phi$
   is consistent with $v$, so it
  induces  a map $\psi: G^{\neq0}\to G$, where $G$ is the value group of $(K,v)$, and the pair $(G,\psi)$ turns out to be an asymptotic couple.
  If $(K,v,\leq,D)$ is a pre-H-field, then $(G,\psi)$ is H-type.
  Note  that if $(K,v,D)$ is a pre-differential-valued field 
  with asymptotic couple $(G,\psi)$ and if $a\in K$, then the pre-differential-valued field $(K,v,aD)$ has
  asymptotic couple $(G,\psi')$, where $\psi'$ is the map $g\mapsto \psi(g)+v(a)$.  
  One can easily check that a pre-differential-valued field has asymptotic integration if and only if its asymptotic couple has it.

  In all the rest of this section, $(K,v,D)$ will be a pre-differential-valued field whose field of constants is $\constants$ and $(G,\psi)$ is its asymptotic couple. 

\subsection{Characterization of the differential rank}\label{diffranksubsec}

 Applying Section \ref{phiranksection} to the special case of pre-differential-valued fields, we introduce the notion of differential rank: 
 \begin{Def}\label{Diffrank}
     The differential rank (respectively, the principal differential rank) of the pre-differential-valued field $(K,v,D)$ is the 
     $\phi$-rank (respectively, the principal $\phi$-rank) of the valued field $(K,v)$, where $\phi$ is the map defined on 
     $K^{\neq0}\backslash\unitsv$ by $\phi(a)=\frac{D(a)}{a}$.     
    \end{Def}
 
   Proposition \ref{fieldgroupchain} then allows us to characterize the differential rank at three different levels: 
  \begin{Thm}\label{diffrankthreelevelsprop}
   Let $(K,v,D)$ be a pre-differential-valued field with asymptotic couple $(G,\psi)$. Then the differential rank (respectively, the principal differential rank) 
of $(K,v,D)$ is equal to the $\psi$-rank (respectively, the principal $\psi$-rank)  of the ordered abelian group $G$. Moreover, if $(G,\psi)$ happens to be H-type, then 
the differential rank (respectively, the principal differential rank) of $(K,v,D)$  is also equal to the $\omega$-rank (respectively, the principal $\omega$-rank) of $\Gamma$, where $\Gamma$ is the value chain of $G$ and 
$\omega$ is the map induced by $\psi$ on $\Gamma$. 
  \end{Thm}

 We now want to express the differential rank as the rank of some quasi-ordered set. This will give us a 
 differential analog of  \cite[Theorem 5.3, Corollary 5.4 and Corollary 5.5]{KuhlmannPointMatusinski}. 
We mentioned in Remark \ref{qoremark} that applying our Proposition \ref{phiqoprop} recovers the results of  \cite{KuhlmannPointMatusinski}, so our idea 
is to apply Proposition \ref{phiqoprop} to the differential case to obtain similar results. One difficulty here is that, even assuming that $(G,\psi)$ is 
H-type, the maps $\phi,\psi$ and $\omega$ are not increasing. We still managed to obtain similar results to those in \cite{KuhlmannPointMatusinski}, as Theorem 
\ref{diffrankcharbyqosThm} below shows. The idea is to remark that, if $\psi$ is H-type and only takes negative values, then $\phi,\psi,\omega$ are increasing, which allows us to apply 
 Proposition \ref{phiqoprop}. If $\psi$ takes non-negative values, one can use Lemma \ref{translatetonegativeLem} which brings back to the case where $\psi$ only takes negative values.\\

  Similarly to what was done in \cite{KuhlmannPointMatusinski},  we define the set $P_K:=K\backslash\ringv$. We now introduce three binary relations 
  $\qophi,\precsim_{\psi},\precsim_{\omega}$ respectively defined on $P_K, G^{<0}$ and $\Gamma$ as follows:\\

  $a\qophi b\Leftrightarrow \exists n,k\in\N_0, v(\phi^n(a))\leq v(\phi^k(b))$\\

	      $g\precsim_{\psi}h\Leftrightarrow \exists n,k\in\N_0,\psi^n(g)\leq\psi^k(h)$\\

	    $\gamma\precsim_{\omega} \delta\Leftrightarrow\exists n,k\in\N_0,\omega^n(\gamma)\leq\omega^k(\delta)$\\

         Three important remarks are in order: first, it is not obvious from their definitions that these relations are quasi-orders, but we will show it in 
  Theorem \ref{diffrankcharbyqosThm}. Secondly, note that it can happen that $\phi(a)\notin P_K$ when $a\in P_K$, which is also a reason why 
we cannot apply Proposition \ref{phiqoprop} directly ($\phi$ was assumed to be a map from $A$ to itself in Section \ref{phiranksection}). Thirdly, 
it can happen that the term $\phi^n(a)$ is not well-defined for a certain $n\in\N$: indeed, remember that the domain of $\phi$ is 
$K^{\neq0}\backslash\unitsv$. Thus, if $v(\phi(a))=0$, $\phi^2(a)$ is not well-defined. Therefore, when we write 
$\phi^n(g)$ it is always implicitly assumed that this expression is defined, i.e expressions like ``$\phi^n(g)\leq\phi^k(h)$'' should be read as 
``$\phi^n(g),\phi^k(g)$ both exist and $\phi^n(g)\leq\phi^k(h)$ holds''. Note that if we allowed the domain of 
$\phi$ to be $K^{\neq0}$ then $\qophi$ would not be a quasi-order, and Theorem \ref{diffrankcharbyqosThm} would fail. 
Similar remarks apply to $\psi$ and $\omega$: $G^{<0}$ is not necessarily stable under $\psi$, $\Gamma$ is not necessarily stable under $\omega$ (we can have 
$\omega(\gamma)=\infty$),
$\psi^2(g)$ is not well-defined if $\psi(g)=0$ and $\omega^2(\gamma)$ is not well-defined if $\omega(\gamma)=\infty$.
 Note however that for any $a\in\dom\phi$, $\phi^n(a)$ is well-defined if and only if $\psi^n(v(a))$ is, and that we then have 
$v(\phi^n(a))=\psi^n(v(a))$. Similarly, for any $g\in G^{<0}$, $\psi^n(g)$ is well-defined if and only if $\omega^n(v_G(g))$ is, in which case 
$v_G(\psi^n(g))=\omega^n(v_G(g))$ holds. 
 As a consequence, we have the following lemma:

\begin{Lem}\label{qothreelevelsLem}
 Assume that $(G,\psi)$ is H-type. 
 For all $a,b\in P_K$, $a\qophi b\Leftrightarrow v(a)\precsim_{\psi}v(b)$.
 For all $g,h\in G^{<0}$, $g\precsim_{\psi}h\Leftrightarrow v_G(g)\precsim_{\omega}v_G(h)$. 
\end{Lem}

\begin{proof}The first statement follows directly from the definitions of $\precsim_{\psi}$ and $\qophi$.
Note however that the image of $\psi$ may contain positive elements, and that $v_G$ reverses the order on $G^{>0}$, so the second 
statement is not trivial. Let us now prove the second statement.
 Let $c$ be a regular cut point for $\psi$. Take $g,h\in G^{<0}$ and set $\gamma:=v_G(g)$ and $\delta:=v_G(h)$. 
 We first show that $g\precsim_{\psi}h\Rightarrow \gamma\precsim_{\omega}\delta$.
 Assume $g\precsim_{\psi}h$ and let $n,k\in\N_0$ with 
$\psi^n(g)\leq\psi^k(h)$. If $\psi^k(h)\leq 0$, then this implies $v_G(\psi^n(g))\leq v_G(\psi^k(h))$, hence 
$\omega^n(\gamma)\leq\omega^k(\delta)$, hence $\gamma\precsim_{\omega}\delta$. Assume that  
$0<\psi^k(h)$. Then by Lemma \ref{positivepsicutpointLem}, $\psi^{k+1}(h)$ is a non-zero cut point for 
$\psi$. By Lemma \ref{lemmacutpoint}(iii), this implies $\psi^{k+1}(h)\sim c$. It follows that $c\neq0$ and $v_G(c)=\omega^{k+1}(\delta)$. 
If $\psi^n(g)=0$, then in particular $n\neq0$. Moreover, 
it follows from Proposition \ref{behaviourofDG}(i)
that $v_G(c)>v_G(\psi^{n-1}(g))$, hence 
$\omega^{n-1}(\gamma)\leq v_G(c)=\omega^{k+1}(\delta)$.
If $v_G(c)>v_G(\psi^n(g))$, then 
$\omega^n(\gamma)\leq\omega^{k+1}(\delta)$. If $0\neq\psi^n(g)$ and $v_G(\psi^n(g))\geq v_G(c)$, then Proposition \ref{behaviourofDG}(i) implies 
$\psi^{n+1}(g)\sim c$, hence 
$\omega^{n+1}(\gamma)=\omega^{k+1}(\delta)$. In any case, we have $\gamma\precsim_{\omega}\delta$. This proves 
$g\precsim_{\psi}h\Rightarrow \gamma\precsim_{\omega}\delta$, let us now prove the converse.
 Assume that $\gamma\precsim_{\omega}\delta$ holds and take 
$n,k\in\N_0$ with $\omega^n(\gamma)\leq\omega^k(\delta)$. This implies 
$v_G(\psi^k(h))\geq v_G(\psi^n(g))$. If $v_G(c)>v_G(\psi^k(h))$, then by Lemma \ref{positivepsicutpointLem}
$\psi^n(g)$ and $\psi^k(h)$ are both negative. Moreover, it follows from the definition of cut point that 
 $v_G(\psi^{k+1}(h))>v_G(\psi^k(h))$, so we have $v_G(\psi^{k+1}(h))>v_G(\psi^n(g))$. Since $\psi^n(g)<0$, this implies 
$\psi^n(g)<\psi^{k+1}(h)$, hence $g\precsim_{\psi}h$. Now assume  that 
$v_G(\psi^k(h))\geq v_G(c)$. If $v_G(c)>v_G(g)$, then $v_G(\psi^k(h))>v_G(g)$, and since $g<0$ this implies $g<\psi^k(h)$, so $g\precsim_{\psi}h$.
Similarly, if $\psi^k(h)\geq0$, then $g<\psi^k(h)$, so $g\precsim_{\psi}h$. Assume then that 
$\psi^k(h)<0$ and $v_G(g)\geq v_G(c)$.  By Proposition \ref{behaviourofDG}(i), $v_G(\psi^k(h))\geq v_G(c)$ and $v_G(g)\geq v_G(c)$ imply
 $\psi^{k+1}(h)\sim c\sim\psi(g)$. If $c=0$, then this implies 
 $\psi^{k+1}(h)=0=\psi(g)$, hence $g\precsim_{\psi}h$. If $c\neq0$, then
 because $(G,\psi)$ is H-type, the relations $\psi^{k+1}(h)\sim c\sim\psi(g)$ imply 
 $\psi^{k+2}(h)=\psi(c)=\psi^2(g)$, hence $g\precsim_{\psi}h$. In any case, we have
 $g\precsim_{\psi}h$.
\end{proof}

    \begin{Thm}\label{diffrankcharbyqosThm}
     Let $(K,v,D)$ be a pre-differential-valued field with asymptotic couple $(G,\psi)$, assume that $(G,\psi)$ is H-type 
and denote by $\omega$ the map induced by $\psi$ on $\Gamma$.
Then the differential rank (respectively, the principal differential rank) of $(K,v,D)$ is equal to the rank (respectively the principal rank) of each one of these q.o sets:
\begin{enumerate}[(1)]
 \item The q.o set $(P_K,\precsim_{\phi})$.
 \item The q.o set $(G^{<0},\precsim_{\psi})$.
 \item The q.o set $(\Gamma,\precsim_{\omega})$.
\end{enumerate}
    \end{Thm}

     \begin{proof}
     Note that thanks to Lemma \ref{qothreelevelsLem}, it is sufficient to prove (3).
Let us show (3). By Theorem \ref{diffrankthreelevelsprop}, the (principal) differential rank of $(K,v,D)$ is equal to the (principal) $\omega$-rank of $(\Gamma,\leq)$. Therefore, we just have to show 
that $\precsim_{\omega}$ is a q.o and that the (principal) rank of $(\Gamma,\precsim_{\omega})$ is equal to the (principal)  $\omega$-rank of $(\Gamma,\leq)$. We first assume that we have $\Psi<0$. In that case,
$\omega$ is an
increasing total map from $\Gamma$ to $\Gamma$. Indeed, $\Psi<0$ clearly implies $\infty\notin\omega(\Gamma)$. Moreover, if $\gamma:=v_G(g),\delta:=v_G(h)\in\Gamma$ and $g,h\in G^{<0}$ are such that 
$\gamma<\delta$, then we have $g<h$. Since $(G,\psi)$ is H-type, it follows that 
$\psi(g)\leq\psi(h)$ and by assumption $\psi(g)$ and $\psi(h)$ are negative, so we must have 
$v_G(\psi(g))\leq v_G(\psi(h))$ i.e $\omega(\gamma)\leq\omega(\delta)$. Thus, we can apply 
Proposition \ref{phiqoprop}, which states that the (principal) $\omega$-rank of $\Gamma$ is equal to the (principal) rank of the q.o set $(\Gamma,\precsim_{\omega})$, so (3) holds.
Now assume that the condition $\Psi<0$ is not satisfied. Let $x,\alpha,\hat{\psi},\hat{\omega}$ be as in Lemma 
\ref{translatetonegativeLem}. We know from Lemma \ref{translatetonegativeLem} that 
$\hat{\Psi}<0$, so we know by what we just proved that the 
(principal) $\hat{\omega}$-rank of $\Gamma$ is equal to the (principal) rank of the q.o set $(\Gamma,\qohom)$. 
Moreover, it follows from  Lemma 
\ref{translatetonegativeLem} and Theorem \ref{diffrankthreelevelsprop}  that 
the (principal) $\hat{\omega}$-rank of $(\Gamma,\leq)$ is equal to the (principal) $\omega$-rank of $(\Gamma,\leq)$. Therefore
it only remains to show  that 
 $\qohom$ and $\precsim_{\omega}$ define the same relation on $\Gamma$.
We use the following claim which follows directly from Lemma \ref{translatetonegativeLem}(iii) and (iv):
\begin{Claim*}
 Let $\gamma\in\Gamma$ and $n\in\N_0$. If $\hat{\omega}^n(\gamma)<\alpha$ or $\omega^n(\gamma)<\alpha$, then 
$\omega^l(\gamma)=\hat{\omega}^l(\gamma)<\alpha$ for every $l\leq n$.
\end{Claim*}

    Assume $\gamma\precsim_{\omega}\delta$ and take $n,k\in\N_0$ with 
$\omega^n(\gamma)\leq\omega^k(\delta)$. If $\omega^k(\delta)<\alpha$ then also $\omega^n(\gamma)<\alpha$ and it follows from the claim that 
$\hat{\omega}^n(\gamma)=\omega^n(\gamma)\leq\omega^k(\delta)=\hat{\omega}^k(\delta)$. 
Assume $\alpha\leq\omega^k(\delta)$. Since $\alpha=\max\hat{\omega}(\Gamma)$, we have 
$\hat{\omega}(\gamma)\leq\alpha$. 
Let $l\leq k$ be minimal with $\alpha\leq\omega^l(\delta)$. If $l=0$ then $\hat{\omega}(\gamma)\leq\delta$. Assume $l\neq0$. 
The claim implies that 
$\hat{\omega}^{l-1}(\delta)=\omega^{l-1}(\delta)$, and since $\omega^l(\delta)\geq\alpha$ it follows from Lemma \ref{translatetonegativeLem}(3) that 
$\hat{\omega}^l(\delta)=\alpha\geq\hat{\omega}(\gamma)$. This proves $\gamma\qohom\delta$. Conversely, assume that 
$\gamma\qohom\delta$ holds, $\hat{\omega}^n(\gamma)\leq\hat{\omega}^k(\delta)$.
If $\hat{\omega}^k(\delta)=\alpha$, then by the claim there must be $l\leq k$ with $\alpha\leq\omega^l(\delta)$. 
If $\gamma<\alpha$ or $\omega^l(\delta)=\infty$ then $\gamma<\omega^l(\delta)$. If 
$\alpha\leq\gamma$ and $\omega^l(\delta)\neq\infty$,  Lemma \ref{translatetonegativeLem}(4) implies $\omega(\gamma)=\omega^{l+1}(\delta)$. Now assume 
$\hat{\omega}^k(\delta)<\alpha$. Then also $\hat{\omega}^n(\gamma)<\alpha$, and it follows from the claim that 
$\omega^n(\gamma)=\hat{\omega}^n(\gamma)\leq \hat{\omega}^k(\delta)=\omega^k(\delta)$. This proves $\gamma\precsim_{\omega}\delta$ and concludes the proof of the Theorem.
     \end{proof}

    In the case of ordered exponential fields, we know that the exponential is compatible with a valuation $w$ if and only if 
$\exp$ induces a map on the residue field $Kw$ (see \cite[Chapter 3, Section 2]{Kuhlmann}). We can then ask if a similar result is true for  pre-differential-valued fields.
We say that $D$ induces a derivation $\bar{D}$ on $Kw$ if $D(\ringw)\subseteq \ringw$ and 
$D(\idealw)\subseteq\idealw$ for all $a,b\in\ringw$. The derivation 
$\bar{D}$ is then defined by $\bar{D}(a+\idealw):=D(a)+\idealw$ (the fact that $\bar{D}$ is a derivation follows directly from its definition). 
 We want to 
characterize the coarsenings $w$ of $v$ such that $D$ induces a derivation on $Kw$.
 The notion of cut point developed above for asymptotic couples 
 plays here an important role, so we extend this notion to fields: If $(K,v,D)$ is a pre-differential-valued field with 
asymptotic couple $(G,\psi)$,  we say that $y\in K$ is a \textbf{cut point} (respectively, a \textbf{regular cut point}) for $(K,v,D)$ if $v(y)$ is a cut point (respectively, a regular cut point) for 
$(G,\psi)$. Such an element always exists thanks to Proposition \ref{behaviourofDG}. We recall that $\quotientval$ denotes the valuation induced by $v$ on $Kw$ (see Section \ref{preliminarysection}).

\begin{Prop}\label{residuefieldprop}
    Let $(K,v,D)$ be a pre-differential-valued field, $y$ a regular cut point for $K$ and $\ringw$ a strict coarsening of $v$. The following holds:
    \begin{enumerate}[(i)]
     \item If $y\notin \ringw$, then  $D$ does not induce a map on $Kw$.
      \item If $y\in\idealw$, and if $D$ induces a derivation on $Kw$, then $D$  induces the constant map $0$ on $Kw$.
      \item  If $y\in \unitsw$, then $D$ induces a non-trivial derivation on $Kw$ making $(Kw,\quotientval,\bar{D})$ a pre-differential-valued field. 
Moreover, if $(K,v,D)$ is a differential-valued field , then  $(Kw,\quotientval,\bar{D})$ is also a differential-valued field.
    \end{enumerate}

\end{Prop}

\begin{proof}
 Set $G_w:=v(\unitsw)$ and $c:=v(y)$. For (i):
  Take $a\in\unitsw\backslash\unitsv$, so $v(a)\in G_w\backslash\{0\}$. By assumption, we have $c<G_w$, hence 
   $v_G(v(a))>v_G(c)$, so by 
Proposition  \ref{behaviourofDG}(ii)  $v(D(a))$ has the same sign and the same 
 archimedean class as $c$. It follows that $v(D(a))<G_w$, hence 
 $D(a)\notin \ringw$.
 For (ii): By assumption, we have $G_w<c$, which by Lemma \ref{lemmacutpoint}(ii), implies $G_w<\psi(c)$. 
 Take $a\in\unitsw$, and let us show that $D(a)\in\idealw$. Assume first that $a\in\ringv$. Then (DV2) implies 
 $v(D(a))>\psi(c)$.  Since $G_w<\psi(c)$, it follows 
 that $G_w<v(D(a))$, hence $D(a)\in\idealw$. Now assume 
 $a\in\unitsw\backslash\ringv$, so in  particular $v(a)\neq0$. 
  Since $a\in\unitsw$, we  have $v(a)\in G_w$. 
  By convexity, $v_G(v(a))>v_G(c)$. By Proposition 
  \ref{behaviourofDG}(ii), 
  $D_G(v(a))$ has the same sign and the same archimedean class as $c$, hence $G_w<D_G(v(a))$. This implies $D(a)\in\idealw$.
For (iii): By assumption, we have $c\in G_w$.
Let $a\in\ringw$ and let us show that $D(a)\in\ringw$. 
Assume first that $a\in\ringv$. 
If $c=0$, then by Lemma \ref{gapiscutpointLem}(ii) we have $0=\sup\Psi$. It then follows from (DV2) that 
$0\leq v(D(a))$,which implies  $D(a)\in\ringw$.
If $c\neq0$, 
then (DV2) implies 
$v(D(a))>\psi(c)$. By Lemma \ref{lemmacutpoint}(ii), $\psi(c)\sim c$, hence 
$\psi(c)\in G_w$ by convexity of $G_w$. It follows that either 
$v(D(a))\in G_w$ or $G_w<v(D(a))$ holds, which implies $D(a)\in\ringw$.
Now assume $a\in\ringw\backslash\ringv$, so $v(a)\neq0$. By Proposition \ref{behaviourofDG}, we either have 
$v(D(a))\sim v(a)$ or $v_G(v(D(a)))\geq v_G(c)$. Since $c\in G_w$ and $v(a)\in G_w$, it follows from the convexity of $G_w$ that 
$v(D(a))\in G_w$, hence $D(a)\in\unitsw$. This shows $D(\ringw)\subseteq\ringw$ and $D(\unitsw\backslash\unitsv)\subseteq\unitsw$.
Now assume that $a\in\idealw$,
 so we have $G_w<v(a)$. By assumption we then have $v_G(c)>v_G(v(a))$, which by Lemma \ref{behaviourofDG}(iv) implies that $v(D(a))$ has same sign and same 
archimedean class as $v(a)$, hence $G_w<v(D(a))$, hence $D(a)\in\idealw$. This shows $D(\idealw)\subseteq\idealw$, which means that 
$D$ induces a derivation $\bar{D}$ on $Kw$. Since $D(\unitsw\backslash\unitsv)\subseteq\unitsw$, $\bar{D}$ is non-trivial.
 The fact that $\bar{D}$ satisfies (DV2) 
follows directly 
from the definition of $\bar{D}$ and $\quotientval$, so $(Kw,\quotientval,\bar{D})$ is a pre-differential-valued field. 
Moreover, the condition 
$\ringv=\constants+\idealv$ clearly implies $\mathcal{O}_{\quotientval}=\constants_{\quotientval}+\mathcal{M}_{\quotientval}$, where 
$\constants_{\quotientval}=\{c+\idealw\mid c\in\constants\}$.
\end{proof}

  In case we start with a pre-H-field, then the induced derivation in Proposition \ref{residuefieldprop}(iii) will also be a pre-H-field:

\begin{Prop}
 Let $(K,v,\leq,D)$ be a pre-H-field and $y\in K$ a regular cut point. If $w$ is a coarsening of $v$ with $y\in\unitsw$, then 
 $(Kw,\quotientval,\leq_w,\bar{D})$ is a pre-H-field. If $(K,\leq,D)$ is a H-field, then so is 
 $(Kw,\leq_w,\bar{D})$.
\end{Prop}
\begin{proof}
 We know from Proposition \ref{residuefieldprop}(iii) that $(Kw,\quotientval,\bar{D})$ is a pre-differential-valued field. It follows from the definitions of 
$\quotientval$ and $\leq_w$ that axioms (PH2) and (H2) are preserved when going to the residue field $Kw$.
Moreover, it directly follows from (PH3) on $K$ that for any 
$a+\idealw\in Kw$, $a+\idealw>_w\mathcal{O}_{\quotientval}\Rightarrow \bar{D}(a+\idealw)\geq_w0$. Now note that, 
in the proof of 
Proposition \ref{residuefieldprop}(iii), we proved that $a\in\ringw\backslash\ringv\Rightarrow D(a)\notin \idealw$.
It follows that $a+\idealw\notin\mathcal{O}_{\quotientval}\Rightarrow \bar{D}(a+\idealw)\neq0$. This proves (PH3) for $Kw$.
\end{proof}

Proposition \ref{residuefieldprop} states that we cannot completely characterize the elements of the differential rank of $(K,v,D)$ by looking at the residue fields; indeed, there may be many 
coarsenings $w$ of $v$ which are not compatible with $\phi$ but still contain $y$, so that $D$ will induce a derivation on $Kw$. Another idea to characterize the differential rank is to look 
at the valued field $(K,w)$, where $w$ is a coarsening of $v$; one can then wonder if $(K,w,D)$ is still a pre-differential-valued field. The following proposition gives us an answer:

\begin{Prop}\label{coarseningprediffvalprop}
 Let $(K,v,D)$ be a pre-differential-valued field (respectively, a pre-H-field), $y$ a regular cut point for $(K,v,D)$ and $w$ a proper coarsening of $v$ such that $y\in\unitsw$. Then $w$ is in the differential rank of 
$(K,v,D)$ if and only if $(K,w,D)$ is a pre-differential-valued field (respectively, a pre-H-field).
\end{Prop}
\begin{proof}
 Set $c:=v(y)$.
 Let $a\in\ringw$ and $b\in\idealw,b\neq0$. Since $b\in\idealw$, we have $D(b)\neq0$, so 
$w(\phi(b))\neq\infty$. 
 If $D(a)=0$ then obviously 
$w(D(a))>w(\phi(b))$.  Thus, we may assume $w(D(a)),w(\phi(b))\neq\infty$. Then the inequality 
$w(D(a))>w(\phi(b))$ is equivalent to $\frac{D(a)}{\phi(b)}\in\idealw$.
 Assume $w$ is not $\phi$-compatible. Since $y\in\unitsw$, it follows from Proposition \ref{behaviourofDG}(i) and (iv)
 that for all $b\in\unitsw$, $\phi(b)\in\unitsw$. Therefore, there must exist $b\notin\unitsw$ with $\phi(b)\in\unitsw$, without loss of generality 
 $b\in\idealw$ (otherwise, take $b^{-1}$). Take an $a\in\unitsw\backslash\unitsv$, 
so $v(a)\in G_w^{\neq0}$. By Proposition \ref{behaviourofDG}(ii), (iii) and (iv), either $D_G(v(a))\sim v(a)$ or 
$v_G(D_G(v(a)))\geq v_G(c)$ is true. By assumption, $c\in G_w$, so the convexity of $G_w$ implies $v(D(a))=D_G(v(a))\in G_w$, hence $D(a)\in\unitsw$.
 We thus have $D(a),\frac{1}{\phi(b)}\in\ringw\backslash\idealw$, and since $\idealw$ is a prime ideal of $\ringw$ this implies 
$\frac{D(a)}{\phi(b)}\notin\idealw$, which contradicts (DV2) for $w$.
Assume now that $w$ is $\phi$-compatible. Take $a\in\ringw$ with $D(a)\neq0$ and $b\in\idealw$. By Proposition \ref{residuefieldprop}(iii), we know that $D(a)\in\ringw$. 
Since $v(b)\notin G_w$ and since $G_w$ is $\psi$-compatible, we know by Lemma
\ref{compatiblecontainsf} that $\psi(v(b))<G_w$, hence
$\phi(b)\notin\ringw$, so $\frac{1}{\phi(b)}\in\idealw$, hence $\frac{D(a)}{\phi(b)}\in\idealw$. This proves (DV2). Now assume that 
$(K,v,\leq,D)$ is a pre-H-field. Since $\ringv$ is $\leq$-convex and $w$ is a coarsening of $v$, $\ringw$ is also $\leq$-convex, so 
$(K,w,\leq,D)$ satisfies (PH2). Since $\ringv\subseteq\ringw$, it also satisfies (PH3).
\end{proof}

\begin{Rem}
 If $w\neq v$, $(K,w,D)$ cannot satisfy (DV1), so it is not a differential-valued field.
\end{Rem}

 Finally, Proposition \ref{residuefieldprop} and \ref{coarseningprediffvalprop} allow us to give a full characterization of coarsenings $w$ of $v$ which belong to the differential rank of 
$(K,v,D)$ by looking simultaneously at  the valued field $(K,w)$ and at the residue field $Kw$:
\begin{Thm}\label{characterizationofdiffrankThm}
Let $(K,v,D)$ be a pre-differential-valued field (respectively, a pre-H-field) and $w$ a coarsening of $v$. Then $w$ is in the differential rank of $(K,v,D)$ if and only if 
the two following conditions are satisfied:
\begin{enumerate}[(1)]
 \item $D$ induces a non-trivial derivation on $Kw$.
 \item $(K,w,D)$ is a pre-differential-valued field (respectively, a pre-H-field).
\end{enumerate}
\end{Thm}

 \begin{proof}
  Take a regular cut point $y$ for $(K,v,D)$ and set $c:=v(y)$. If $w$ is $\phi$-compatible, then by Lemma \ref{compatiblecontainsf} we must have $c\in G_w$ hence $y\in\unitsw$, which by Proposition 
\ref{residuefieldprop} implies that $D$ induces a non-trivial derivation on $Kw$. We can then apply Proposition \ref{coarseningprediffvalprop} and we get that $(K,w,D)$ is 
a pre-differential-valued field (respectively, a pre-H-field). Conversely, assume (1) and (2) hold. By Proposition \ref{residuefieldprop}, (1) implies $y\in\unitsw$, so we can apply \ref{coarseningprediffvalprop} and we get that, 
since (2) holds, $w$ must be in the differential rank of $(K,v,D)$.
 \end{proof}

  \subsection{The unfolded differential rank}\label{unfrksection}
  Our definition of the differential rank is not quite satisfactory if the cut point for $\psi$ is not $0$. Indeed, by 
  Lemma \ref{compatiblecontainsf}, we see
  that the $\psi$-rank of $G$ does not give any information on what happens for ``small'' elements, 
  i.e elements $g$ with $v_G(g)\geq v_G(c)$ where $c$ is a cut point 
  for $\psi$.
  We need to ``unfold'' the map $\psi$ around $0$ to 
  get this information, i.e we need to translate $\psi$ in order to obtain a new map whose cut point is closer to $0$ than $c$ is, thus 
  gaining information 
  on the behavior of $\psi$ around $0$.  Ideally, this translate of $\psi$ should have $0$ as a cut point.
  If $\psi$ happens to have a gap or a maximum $g$, then we can do this by considering the map $\psi':=\psi-g$, since this map has $0$ as
  the cut point. However, 
  if $(G,\psi)$ has asymptotic integration, then we cannot obtain a map with cut point $0$ by simply 
  translating $\psi$. Instead, we need to consider an infinite family of translates of $\psi$ whose cut points approach $0$, and then take the 
  union of their ranks.

   Now let us denote by $R$ the $\psi$-rank of $G$ and by $P$ the principal $\psi$-rank of $G$.
    For any $g\in G^{\neq0}$, let us denote by $\psi_g$ the map $G^{\neq0}\to G$, $h\mapsto\psi(h)-\psi(g)$.  Note that $(G,\psi_g)$ is also an 
    asymptotic couple. For every $g$ we choose 
    a cut point $c_g$ for $\psi_g$.
    We denote by $S_g$ the  $\psi_g$-rank of $G$ and by $Q_g$ the   $\psi_g$-principal rank of $G$.



  \begin{Lem}\label{cgarbitrarilysmall}
   Let  $h\in G^{\neq0}$. For any cut point $c_h$ for $\psi_h$, we have 
  $v_G(c_h)\geq v_G(h)$. 
  \end{Lem}
 
\begin{proof}
 If $c_h=0$ this is clear, so assume $c_h\neq0$. By Lemma 
 \ref{lemmacutpoint}(ii), we have $c_h\sim\psi_h(c_h)$.
 By Lemma \ref{Fundlemofascouples}, we have $v_G(\psi_h(c_h))=v_G(\psi(c_h)-\psi(h))>v_G(c_h-h)$. It follows that 
 $v_G(c_h)>v_G(c_h-h)$, which implies $v_G(c_h)\geq v_G(h)$.
\end{proof}

Lemma \ref{cgarbitrarilysmall} shows in particular that we can choose $g$ so that $c_g$ is arbitrarily small, which means that the family 
 $\{\psi_g\}_{g\in G}$ is well-suited for our purpose. This motivates the following definition: 

\begin{Def}\label{unfrankdef}
   The \textbf{unfolded $\psi$-rank} of the asymptotic couple $(G,\psi)$ is the order-type of the totally ordered set 
   $S:=\bigcup_{0\neq g\in G}S_g$.  The \textbf{principal unfolded $\psi$-rank} of the asymptotic couple $(G,\psi)$ is the order-type of the totally ordered set 
   $Q:=\bigcup_{0\neq g\in G}Q_g$. If $(K,v,D)$ is a pre-differential-valued field, we define its 
   \textbf{unfolded differential rank} ( respectively, its \textbf{ principal unfolded differential rank}) as the unfolded $\psi$-rank of $G$ (respectively the principal unfolded $\psi$-rank of $G$), 
where $(G,\psi)$ is the asymptotic couple 
   associated to $(K,v,D)$. \comment{We say that a convex subring $\ringw$ of $(K,v)$ is in the unfolded differential rank of 
   $(K,v,D)$ if $G_w$ is in the unfolded $\psi$-rank of $G$.}
 \end{Def}
 
 In order to justify Definition \ref{unfrankdef}, 
 we still need to check that  $S$ and $P$ satisfy the conditions that we want. We want to show that 
 $S$ contains the $\psi$-rank of $G$, and that the only new subgroups that were added in the process 
 are subgroups contained in $\{g\in G\mid v_G(g)\geq v_G(c)\}$, where $c$ is a cut point for $\psi$.

\begin{Lem}\label{unfranklem}
 Let $g\in G^{\neq0}$. The following holds: 
 
 \begin{enumerate}[(i)]
  \item We have $S_g=\{H\in S\mid g\in H\}$. In particular, $S_g$ is a final segment of $S$.
  \item For any convex subgroup $H$ of $G$, $H\in Q$ if and only if there is $h\in G$ such that $H=\underset{F\in S_h}{\bigcap}F$.
  \item We have $Q_g=\{H\in Q\mid g\in H\}$. In particular, $Q_g$ is a final segment of $Q$.
   
 \end{enumerate}

\end{Lem}

\begin{proof}
Let us prove (i). If $H\in S_g$, then obviously $H\in S$. Moreover, we have 
 $\psi_g(g)=0\in H$, so if $H\in S_g$ it follows from $\psi_g$-compatibility that $g\in H$. 
 This proves $S_g\subseteq\{H\in S\mid g\in H\}$. Now assume that $g\in H\in S$ holds. 
 By definition of $S$, there is $h\in G$ such that $H\in S_h$. 
 Let us show that 
 $H\in S_g$. Let $f\in G^{\neq0}$. We already showed that $H\in S_h$
  implies $h\in H$. By definition of $S_h$, we have 
 $f\in H\Leftrightarrow \psi_h(f)\in H$.
 Moreover, we have  $\psi_h(f)=\psi_g(f)+\psi(g)-\psi(h)$.
    By Lemma \ref{Fundlemofascouples}, we have $v_G(\psi(g)-\psi(h))>v_G(g-h)$. Since $g,h\in H$, 
    it follows from convexity of $H$ that $\psi(g)-\psi(h)\in H$. It follows that 
   $\psi_g(f)\in H$ if and only if $\psi_h(f)\in H$. It follows that 
   $f\in H\Leftrightarrow\psi_g(f)\in H$. This shows $H\in S_g$.   
   Let us show (ii). By definition of $Q$, $H\in Q$ if and only if there exists $h,f\in G$ such that 
   $H=\underset{f\in F\in S_h}{\bigcap}F$. Now note that, for any  $F\in S$, it follows from (i) that 
   $f\in F\in S_h\Leftrightarrow f,h\in F\in S\Leftrightarrow h\in F\in S_h$, so 
   $H=\underset{f\in F\in S_h}{\bigcap}F$ if and only if $H=\underset{h\in F\in S_f}{\bigcap}F$.
   Therefore, we can assume without loss of generality that $v_G(f)\geq v_G(h)$. 
   It follows from (i) that $f\in F\in S_h\Rightarrow h\in F\in S_h$, and since $v_G(f)\geq v_G(h)$ it follows from 
   convexity that $h\in F\in S_h\Rightarrow f\in F\in S_h$. It follows that 
   $H=\underset{f\in F\in S_h}{\bigcap}F$ holds if and only if 
   $H=\underset{h\in F\in S_h}{\bigcap}F$. By (i), we have 
   $\underset{h\in F\in S_h}{\bigcap}F=\underset{F\in S_h}{\bigcap}F$. This shows (ii).    
  Now let us show (iii). Assume $H\in Q_g$. Then in particular, $H\in S_g$, hence $g\in H$ by (i).  
  Assume $g\in H\in Q$. By (ii), there is $h$ such that 
  $H=\underset{F\in S_h}{\bigcap}F$. Because $g\in H$, we have 
  $H=\underset{g\in F\in S_h}{\bigcap}F$. 
  It follows from (i) that, for any  $F\in S$, 
  $g\in F\in S_h\Leftrightarrow g,h\in F\in S\Leftrightarrow h\in F\in S_g$. 
  It follows that $H=\underset{h\in F\in S_g}{\bigcap}F$. This means that $H$ is the smallest element of $S_g$ with $h\in H$, hence 
  $H\in Q_g$.  
\end{proof}

The principal unfolded differential rank is related to the unfolded differential rank the same way that 
principal ranks are usually related to the corresponding rank:

\begin{Prop}\label{charofQProp}
 If $H$ is a convex subgroup of $G$, then $H\in Q$ if and only if there exists $h\in G$ such that 
   $H$ is the smallest element of $S$ containing $h$.
\end{Prop}

\begin{proof}
  Let $H\in Q$. It follows from Lemma \ref{unfranklem}(ii) that $H=\underset{F\in S_g}{\bigcap}F$ for some 
  $g\in G$.
  It then follows from Lemma \ref{unfranklem}(i) that $H=\underset{g\in F\in S}{\bigcap}F$, i.e $H$ is the smallest element of $S$ 
  containing $g$. Conversely, assume that 
  $H$ is the smallest element of $S$ 
  containing $g$. It follows from Lemma \ref{unfranklem}(i) 
  that $H$ is then the smallest element of $S_g$ containing $g$, 
  hence $H\in Q_g\subseteq Q$.  
\end{proof}

 Now we can describe the connection between the (principal) $\psi$-rank and the (principal) unfolded
 $\psi$-rank:
 
 \begin{Prop}\label{unfrankProp}
  Let $c$ be a cut point for $\psi$. The following holds:
  \begin{enumerate}[(i)]
   \item We have $R=\{ H\in S\mid c\in H\}$.  In particular, $R$ is a final segment of $S$.
   
   \item We have $P=\{H\in Q\mid c\in H\}$. In particular, $P$ is a final segment of $Q$.
   \item Assume $c\neq0$ and let $G_c$ be the $\psi$-principal subgroup of $G$ generated by 
    $c$. Then $R$ (respectively, $P$)  is the principal final segment of 
    $S$ (respectively, of $Q$) generated by $G_c$.
      \item Assume $c=0$.  Then $R=S$ and $P=Q$.
  
   \item Assume $g$ is either the maximum of $\Psi$ or a gap for $\psi$. Then $S$
   (respectively, $Q$) is equal to the 
 $\psi'$-rank of $G$ (respectively, the principal $\psi'$-rank of $G$), where $\psi':=\psi-g$.
  \end{enumerate}

 \end{Prop}
 \begin{proof}
   Let us prove (i). It follows from Lemma \ref{compatiblecontainsf} that $H\in R\Rightarrow c\in H$. Now   
   we just have to show that, for any non-trivial convex subgroup 
   $H$ of $G$ containing $c$, $H\in R\Leftrightarrow H\in S$.
   Let $H$ be a convex subgroup of $G$ containing $c$.
   Take $g\in H^{\neq0}$. It follows from 
   Lemma \ref{unfranklem}(i) that 
   $H\in S$ if and only if $H\in S_g$.   
   By Lemma \ref{cinHlemma}, we have
   $\psi(g)\in H$. For any $f\in G^{\neq0}$, we have
    $\psi_g(f)=\psi(f)-\psi(g)$, which implies 
    $\psi_g(f)\in H\Leftrightarrow \psi(f)\in H$. It then follows from the definition of 
    $R$ and $S_g$ 
    that 
    $H\in R$  if and only if $H\in S_g$.
    This shows that $H\in R$ if and only if $H\in S$. Now let us show (ii).
    Assume $H\in P$. Then there exists $g\in G$ such that $H$ is the smallest element of $R$ containing $g$.
    It follows from (i) that $H$ is the smallest element of $S$ containing  
    $c$ and $g$. It then follows from convexity that $H$ is the smallest element of $S$ containing $h$, where we set 
    $h:=c$ if $v_G(g)\geq v_G(c)$ and $h:=g$ is $v_G(c)>v_G(g)$. By Proposition \ref{charofQProp}, this implies 
    $H\in Q$. Conversely, assume $c\in H\in Q$. Then by Proposition \ref{charofQProp}, there is $g\in H$ such that 
    $H$ is the smallest element of $S$ containing $g$. Since $c\in H$, it follows from (i) that 
    $H$ is the smallest element of $R$ containing $g$, hence $H\in P$.
   This shows (ii). (iii) and (iv) then follow directly from (i) and (ii). Let us prove (v).
 Note that $\psi_g(h)=\psi'(h)-\psi'(g)=\psi_g'(h)$ for any $g,h$. It follows that the (principal) unfolded $\psi$-rank of $G$ is equal to the (principal) unfolded 
$\psi'$-rank of $G$. By lemma \ref{translatebygap}, $0$ is a cut point for $\psi'$. It then follows from (iv) that 
the (principal) unfolded $\psi'$-rank of $G$ is equal to the (principal)
$\psi'$-rank of $G$.
 \end{proof}

 \begin{Rem}
  Proposition \ref{unfrankProp} shows that the unfolded differential rank has the desired properties. 
  Indeed, (i) and (ii) show that the (principal) differential rank is contained in the (principal) unfolded differential rank and that 
  the only subgroups which were added in the process are groups which do not contain $c$.
  Moreover, (v) shows that taking the unfolded differential rank generalizes the idea of translating $\psi$ by a gap.
 \end{Rem}

 For a pre-differential-valued field $(K,v,D)$ with asymptotic couple $(G,\psi)$, we say that a coarsening $w$ of $v$ lies in the unfolded differential rank 
 of $(K,v,D)$ if $G_w$ lies in unfolded $\psi$-rank of $G$. We can give an analog of Theorem \ref{characterizationofdiffrankThm} for the unfolded differential rank, which characterizes 
 the convex subrings of $K$  lying in the unfolded differential rank:

 \begin{Thm}\label{charofunfrk}
  Let $(K,v,D)$ be a pre-differential-valued field and $w$ a coarsening of $v$. Then $w$ is in the unfolded differential rank  of $(K,v,D)$
if and only if $(K,w,D)$ is a pre-differential-valued field.
 \end{Thm}

  \begin{proof}
    Assume $w$ is in the unfolded differential rank of $(K,v,D)$. That means there is $g\in G$ such that 
    $G_w$ is in the $\psi_g$-rank of $G$. Now take $a\in K$ with $v(a)=\psi(g)$. $(K,v,aD)$ is a pre-differential-valued 
    field with asymptotic couple $(G,\psi_g)$ and $\ringw$ is in the differential rank of 
    $(K,v,aD)$. By Theorem \ref{characterizationofdiffrankThm}, it follows that $(K,w,aD)$ is a pre-differential valued field, so $(K,w,D)$ is also 
    a pre-differential-valued field. Conversely, assume $(K,w,D)$ is  a pre-differential-valued field, $w\neq v$,
    and take $a\in\unitsw\backslash\ringv$. Set $g:=v(a)$. By Lemma \ref{cgarbitrarilysmall}, we have 
    $v_G(c_g)\geq v_G(g)$, so by convexity $c_g\in G_w$. By Proposition \ref{residuefieldprop}(3) and Theorem \ref{characterizationofdiffrankThm}, it follows that 
    $\ringw$ is in the differential rank of $(K,v,aD)$, so $G_w$ is in the $\psi_g$-rank of 
    $G$, so it is in the unfolded $\psi$-rank of $G$.
  \end{proof}

 Finally, we want to connect the unfolded differential rank with the exponential rank of 
 exponential ordered fields.  
 We recall the following definition from  \cite[Section 3]{Kuhlmanncontraction}: 
 
 \begin{Def}\label{contractiondef}
  A \textbf{centripetal precontraction} on $G$ is a map $\chi: G\to G$ 
  satisfying the following conditions for all $g,h\in G$: 
 \begin{enumerate}[(i)]
 \item $\chi(g)=0\Leftrightarrow g=0$.
 \item $\chi$ preserves $\leq$.
 \item $\chi(g)=-\chi(-g)$.
 \item If $v_G(g)=v_G(h)$ and $g,h$ have the same sign, 
 then $\chi(g)=\chi(h)$.
 \item $\vert \chi(g)\vert<\vert g\vert$.
\end{enumerate}
 \end{Def}

 As  was already noted in \cite[Section 5]{Asch}, asymptotic couples are related to precontractions.
Indeed, assume that $(G,\psi)$ has asymptotic integration; then we can define $\int g:=D_G^{-1}(g)$ for any $g\in G$ (note that $D_G$ is injective because 
it is strictly increasing). We can then define the map:
$\chi(g):G^{<0}\to G, g\mapsto\int\psi(g)$. We extend this map to $G$ by setting $\chi(0):=0$ and $\chi(g)=-\chi(-g)$ for every $g>0$. If $(G,\psi)$ is H-type, it is then easy 
 to see that $\chi$ is a centripetal precontraction. If $(G,\psi)$ is not H-type, then $\chi$ is not  a precontraction because axiom 
(ii) above is not satisfied, but 
the following result still holds:

\begin{Prop}\label{exprank}
 If $(G,\psi)$ has asymptotic integration, then the unfolded $\psi$-rank (respectively, the principal unfolded $\psi$-rank) of $G$ coincides with the $\chi$-rank (respectively, the principal $\chi$-rank) of $G$.
\end{Prop}

\begin{proof}
 We start by showing the following claim:
 \begin{Claim*} For any $g,h\in G^{\neq0}$, if $v_G(g)\geq v_G(\chi(h))$ then $\psi_g(h)\sim\chi(h)$.
\end{Claim*}
 \begin{proof}
   Because $\chi(-h)=-\chi(h)$ and $\psi_g(h)=\psi_g(-h)$, it is sufficient to consider the case 
   $h<0$.
   By definition of $\chi$, we have $\chi(h)+\psi(\chi(h))=\psi(h)$, hence 
$\chi(h)=\psi(h)-\psi(\chi(h))=\psi_g(h)+\psi(g)-\psi(\chi(h))$ hence $\chi(h)-\psi_g(h)=\psi(g)-\psi(\chi(h))$.
  By Lemma \ref{Fundlemofascouples}, we have 
$v_G(\psi(g)-\psi(\chi(h)))>v_G(g-\chi(h))$, so if $v_G(g)\geq v_G(\chi(h))$ we have $v_G(\chi(h)-\psi_g(h))=v_G(\psi(g)-\psi(\chi(h)))> v_G(\chi(h))$ which implies 
$\chi(h)\sim\psi_g(h)$.
\end{proof}
 
Now let us prove the proposition. 
 Let $H$ be a convex subgroup of $G$. Assume $H$ is $\chi$-compatible and take $g\in H$, $g\neq0$. 
By Lemma \ref{cgarbitrarilysmall}, $v_G(c_g)\geq v_G(g)$, hence $c_g\in H$ by convexity. Now take $h\in G$ such that $\psi_g(h)\in H$. If 
$v_G(\chi(h))\geq v_G(g)$ then by convexity $\chi(h)\in H$; if 
$v_G(g)\geq v_G(\chi(h))$ the claim implies $\chi(h)\sim\psi_g(h)$, hence by convexity $\chi(h)\in H$. In any case we have $\chi(h)\in H$, and since $H$ is 
$\chi$-compatible it follows that $h\in H$. We proved $c_g\in H$ and $\psi_g(h)\in H\Rightarrow h\in H$, so by 
Lemma \ref{compatiblecontainsf} $H$ is $\psi_g$-compatible, which means that $H$ is in the unfolded $\psi$-rank of $G$.
Conversely, assume $H\in S$ and take $h\in G$. Since $H\in S$, there is $g\in G$ with $H\in S_g$ and by Lemma \ref{unfranklem}(i) we may assume 
$v_G(g)\geq v_G(\chi(h))$. By the claim, we then have $\psi_g(h)\in H\Leftrightarrow \chi(h)\in H$. Since $H\in S_g$, this implies 
$h\in H\Leftrightarrow \chi(h)\in H$, so $H$ is $\chi$-compatible. This proves that $S$ is the set of $\chi$-compatible convex subgroups of $G$.
It then follows from Proposition \ref{charofQProp} that $Q$ is the set of $\chi$-principal 
convex subgroups of $G$.
\end{proof}
 
 Proposition \ref{exprank} yields an immediate corollary for exponential fields. We refer to \cite{Kuhlmann} for the definition of $v$-compatible $(GA),(T_1)$-exponential:
 
 \begin{Cor}\label{unfoldedrankisexponentialrank}
  Let $(K,\leq,D)$ be a H-field, let $v$ be the natural valuation associated to $\leq$ and assume that there exists a $v$-compatible $(GA),(T_1)$-exponential $\exp$ on $K$ such that 
$\phi(\exp(a))=D(a)$ for any $a\in K$. Then the exponential rank of $(K,\leq,\exp)$ is equal to the 
unfolded differential rank of $(K,\leq,D)$.
 \end{Cor}

 \begin{proof}
    We know from \cite[Theorem 3.25]{Kuhlmann} that $\log=\exp^{-1}$ induces a map $\chi$ on $G$ and that the exponential rank of 
  $K$ is equal to the $\chi$-rank of $G$ (note that this also follows from our Proposition \ref{fieldgroupchain}). With our assumption, it is easy to check that 
   $\chi(g)+\psi(\chi(g))=\psi(g)$ for all $g\in G^{<0}$. By definition of 
   $D_G$, this implies $D_G(\chi(g))=\psi(g)$, i.e $\chi(g)=\int\psi(g)$ for all $g\in G^{<0}$. Moreover, one can easily check that 
   $\chi(g)=-\chi(-g)$ for all $g\in G$. Therefore, $\chi$ coincides with the map $\chi$ in Proposition \ref{exprank}.
   The claim then follows from Proposition \ref{exprank}.
 \end{proof}

 \section{Derivations on power series}\label{powerseries}

   The goal of this Section is first to answer Question \ref{question} of the introduction, which then allows us to partially answer
  Question \ref{question1}. We then apply our results to construct a derivation on a field of generalized power series so that we obtain a H-field of 
given principal differential rank and principal unfolded differential rank. 
If $(K,v,D)$ is a differential-valued field, then we say that $D$ is a \textbf{H-derivation} if the asymptotic couple associated 
to $(K,v,D)$ is H-type. In this section, we will only consider H-derivations. Therefore,
all the asymptotic couples appearing in this section are H-type.
Section \ref{derivationfromacsection} answers Question \ref{question} and Section \ref{Hderivationsection} answers a variant of Question \ref{question1} where 
$D$ is required to be a H-derivation. Section 
\ref{Hardytypesection} answers another variant of Question \ref{question1} where we require $D$ to be of Hardy type
(see Definition \ref{Hardytypederivation}),
 which is connected to  the work done in 
\cite{KuhlmannMatusinski} and \cite{KuhlmannMatusinski2}.

We want to recall a few facts about the theory of valued groups. Let $v:G\to\Gamma\cup\{\infty\}$ be a valuation.
For all $\gamma\in\Gamma$, 
    the sets $(G_v)^{\gamma}:=\{g\in G\mid v(g)\geq\gamma\}$ and $(G_v)_{\gamma}:=\{g\in G\mid v(g)>\gamma\}$ are subgroups of $G$. The quotient groups 
$(B_v)_{\gamma}:=(G_v)^{\gamma}/(G_v)_{\gamma}$, $\gamma\in\Gamma$, are called the \textbf{components} of the valued group $(G,v)$. 
 The pair $(\Gamma,((B_v)_{\gamma})_{\gamma\in\Gamma})$ is called the 
\textbf{skeleton} of the valued group $(G,v)$.
  Given an ordered set $\Gamma$ and a family of groups 
  $(B_{\gamma})_{\gamma\in\Gamma}$, we define the \textbf{Hahn product} of the family 
  $(B_{\gamma})_{\gamma\in\Gamma}$, denoted $\Hahn_{\gamma\in\Gamma}B_{\gamma}$,
as the valued group $(G,v)$, where 
  $G:=\{g=(g_{\gamma})_{\gamma\in\Gamma}\in\prod_{\gamma\in\Gamma}B_{\gamma}\mid\supp(g)\text{ is well-ordered}\}$ and $v(g):=\min\supp(g)$. Here $\supp(g)$ 
  denotes the support of $g$, i.e $\supp((g_{\gamma})_{\gamma\in\Gamma})=\{\gamma\in\Gamma\mid g_{\gamma}\neq0\}$.
  $(G,v)$ is a valued group with skeleton $(\Gamma,(B_{\gamma})_{\gamma\in\Gamma})$. 
 It is usual to denote elements of $\Hahn_{\gamma\in\Gamma}B_{\gamma}$ as formal sums 
$\sum_{\gamma\in\Gamma}g_{\gamma}$, where $g_{\gamma}\in B_{\gamma}$ for all $\gamma\in\Gamma$.
  We recall Hahn's embedding theorem:

\begin{Thm}[Hahn's embedding theorem]\label{Hahnsembeddingtheorem}
 Let $(G,v)$ be a divisible valued group satisfying the condition $v(ng)=g$ for all $g\in G$ and $n\in\Z$ with $n\neq0$. Denote by $(\Gamma,(B_{\gamma})_{\gamma\in\Gamma})$
  the skeleton of the valued group $(G,v)$. Then $(G,v)$ is embeddable as a valued group into 
  the Hahn product $\Hahn_{\gamma\in\Gamma}B_{\gamma}$.
\end{Thm}

Throughout this section, $k$ will denote a field, $G$ an ordered abelian group and $K:=k((G))$ the field of generalized power series generated by $k$ and $G$. 
We denote by $v$ the usual valuation on $K$, i.e. $v(a)=\min\supp(a)$.
 If $k$ is ordered, then  $K$ is also an ordered field, with the order of $K$ being defined as follows:
 we say that $a< b$ if and only if $a_g< b_g$, where $g=v(a-b)$.

\subsection{Defining a derivation of given asymptotic couple}\label{derivationfromacsection}    
    
   Let $k$ be a field and let $(G,\psi)$ be a H-type asymptotic couple.
    We would like to define a derivation $D$ on 
   $K$ making $(K,v,D)$ 
   a pre-differential-valued  field whose associated asymptotic couple is $(G,\psi)$.  If $k$ happens to be ordered,
   then we  want $(K,\leq,D)$ 
to be a H-field.  The idea is to view the map $\psi:G\to\Psi\cup\{\infty\}$ as a valuation ($\psi$ is extended to $G$ by setting $\psi(0):=\infty$). 
 We first assume that the valued group $(G,\psi)$ is 
maximally valued, i.e $(G,\psi)$ is isomorphic as a valued group to the Hahn product  $H:=\mathrm{H}_{\lambda\in\Psi}\comppsi$. We will later see that this condition is actually not essential.
 We shall write elements of $G$ as sums $g:=\sum_{\lambda\in\Psi}g_{\lambda}$ with $g_{\lambda}\in \comppsi$.  Note that $\psi(g)=\min\{\lambda\in\Psi\mid\lambda\in\supp(g)\}$. We shall write 
 $B_{\lambda}$ instead of $\comppsi$ for every $\lambda\in\Psi$.

   To define a derivation on 
      $K$ we proceed in three steps:
      Step 1: define $D$ on the fundamental monomials, i.e define $D(t^{g_{\lambda}})$ for each $g_{\lambda}\in B_{\lambda}$ 
      for every $\lambda\in\Psi$.
      Step 2: extend $D$ to all monomials by using a strong Leibniz rule.
      Step 3: extend $D$ to $K$ by strong linearity.      
      This idea is inspired by the work in \cite{KuhlmannMatusinski}. In \cite{KuhlmannMatusinski},
      the authors assumed that the map $D$ was already given on the fundamental monomials and gave
      conditions for this map to be extendable to the whole field. They only do this in the particular case where $G$ 
      is a Hahn product of copies of $\R$. Here we do it in a more general setting since we do not make any assumption 
      on $(G,\psi)$ except that it is a H-type asymptotic couple; moreover,
       we define $D$ explicitly on the 
      fundamental monomials. The idea for step 1 comes from the following remark: if 
      $(K,v,D)$ is a pre-differential-valued field, then for any $a\in K$ we have 
      $v(D(a))=v(a)+\psi(v(a))$. Therefore, we naturally want to define $D(t^{(g_{\lambda})})$ as an element with valuation 
      $g_{\lambda}+\lambda$  (note that this expression makes sense because $\lambda\in\Psi\subseteq G$). Assume that a family $u_{\lambda}:B_{\lambda}\to k$ ($\lambda\in\Psi$) of group homomorphisms has been given. Define then 
      $D(t^{g_{\lambda}})=u_{\lambda}(g_{\lambda})t^{g_{\lambda}+\lambda}$. 
Note that a similar idea was already used 
      in \cite{Aschdries3}, but only in the case where $G$ is divisible and admits a valuation basis, 
      which is a strong restriction. Now we use the strong Leibniz rule to extend $D$ to the set of all monomials of $K$: let 
       $g=\sum_{\lambda\in\Psi}g_{\lambda}\in G$; the strong Leibniz rule implies:

      \[D(t^g):=t^g\sum_{\lambda\in\Psi}\frac{D(t^{g_{\lambda}})}{t^{g_\lambda}}=t^g\sum_{\lambda\in\Psi}u_{\lambda}(g_{\lambda})t^{\lambda}.\]

      Note that the support of the family $(u_{\lambda}(g_{\lambda})t^{\lambda})_{\lambda\in\Psi}$ is included 
      in the support of $g$ so it is well-ordered, which proves that $D(t^g)$ is indeed an element of $K$. Note also that since the support of 
$0$ is $\varnothing$ we have $D(t^0)=0$.
Finally, we apply strong linearity to extend $D$ to all elements of $K$: for 
      $a=\sum_{g\in G}a_gt^g\in K$, we define: 
      \[\tag{\dag}\label{defofder}D(a):=\sum_{g\in G}a_gD(t^g)=\sum_{0\neq g\in G}\sum_{\lambda\in\Psi}a_gu_{\lambda}(g_{\lambda})t^{g+\lambda}.\] 

      Because formula (\ref{defofder}) is an infinite sum, it is not clear that this expression is well-defined. In order to make sense 
      of formula (\ref{defofder}), we recall the notion of summability introduced in \cite{Fuchs}.       
Let $(a_i)_{i\in I}$ be a family of elements of $K$ with 
$a_i=\sum_{g\in G}a_{i,g}t^g$. We say that the family $(a_i)_{i\in I}$ is \textbf{summable} if 
the following two conditions are satisfied: 
\begin{enumerate}[(i)]
 \item The set $A:=\bigcup_{i\in I}\supp(a_i)$ is well-ordered
 \item for any $g\in A$, the set $A_g:=\{i\in I\mid a_{i,g}\neq0\}$ is finite.
\end{enumerate}
If the family $(a_i)_{i\in I}$ is summable, we define its sum as $\sum_{g\in G}b_gt^g\in K$, where $b_g=0$ if 
$g\notin A$ and $b_g=\sum_{i\in A_g}a_{i,g}$ if $g\in A$.

      Now let $a=\sum_{g\in G}a_gt^g$. We want to prove
      that the 
      family $(a_gD(t^g))_{g\in \supp(a)}$ is summable.
      We set $A:=\bigcup_{g\in \supp(a)}\supp(D(t^g))=\{g+\lambda\mid g\in \supp(a),\lambda\in \supp(g)\}$ and 
      for any $g\in \supp(a)$ and any $\lambda\in \supp(g)$,
    $A_{g,\lambda}:=\{(h,\mu)\mid h\in \supp(a),\mu\in \supp(h), g+\lambda=h+\mu\}$.
      The key to summability is the following fact (see \cite[Proposition 2.3(1)]{Aschdries2}):
      
       \begin{Lem}\label{keytosummability}
    
    For any $\lambda,\mu\in\Psi$ with $\mu\neq\lambda$, $\psi(\lambda-\mu)>\min(\lambda,\mu)$.
   \end{Lem}
   
   The next two lemmas will help us prove the summability of $(a_gD(t^g))_{g\in \supp(a)}$:

   \begin{Lem}\label{Fundlemsummability}
    Let $g,h\in G$ with $g\leq h$ and $\lambda,\mu\in\Psi$ such that $h+\mu\leq g+\lambda$. Then 

$\mu\in \supp(g)\Leftrightarrow \mu\in \supp(h)$.
   \end{Lem}

   \begin{proof}
    The case $g=h$ is trivial so assume $g<h$. Then we must have $\mu<\lambda$. Moreover, $0<h-g\leq \lambda-\mu$, hence (since $\psi$ is H-type)
    $\psi(h-g)\geq \psi(\lambda-\mu)$, hence by \ref{keytosummability} $\psi(g-h)>\min(\lambda,\mu)=\mu$. Now if $\mu$ were in $\supp(g)$ but not in 
    $\supp(h)$ then $\mu$ would also be in $\supp(g-h)$, hence $\psi(g-h)\leq\mu$, which is a contradiction. This proves the Lemma.
   \end{proof}

    \begin{Lem}\label{sequencesLem}
     If $(g_n+\lambda_n)_n$ is a decreasing sequence in $A$, then $(g_n)_n$ cannot be strictly 
     increasing. If $(g_n)_n$ is constant, then $ (g_n+\lambda_n)_n$ is not strictly decreasing.
    \end{Lem}

    \begin{proof}    
     Assume that $(g_n+\lambda_n)_n$ is  decreasing and $(g_n)_n$ increasing. Then for all $n\in\N$ we have 
     $g_0\leq g_n$ and $g_n+\lambda_n\leq g_0+\lambda_0$, so by  Lemma \ref{Fundlemsummability} 
     $\lambda_n\in \supp(g_0)$. Since $\supp(g_0)$ is well-ordered, it follows that $(\lambda_n)_n$ cannot be strictly decreasing. Since $(g_n+\lambda_n)_n$ is  decreasing, it follows that 
$(g_n)_n$ cannot be strictly increasing. Moreover, if $(g_n)_n$ is constant, then $(g_n+\lambda_n)_n$ cannot be strictly decreasing.       
    \end{proof}

    \begin{Prop}
     The family $(a_gD(t^g))_{g\in \supp(a)}$ is  summable.
    \end{Prop}

   \begin{proof}
   Assume there exists a strictly decreasing sequence $(g_n+\lambda_n)_n$ in $A$. Without loss of generality we can 
   assume that $(g_n)_n$ is either constant, strictly decreasing or strictly increasing. Since $g_n\in\supp(a)$ for all $n$, $(g_n)_n$ cannot be strictly decreasing, 
  so without loss of generality $(g_n)_n$ is either constant or strictly increasing. This contradicts Lemma \ref{sequencesLem}, and it follows that $A$ is well-ordered.
 Now let 
   $h+\mu\in A$ and assume there is an infinite subset $\{(g_n,\lambda_n)\mid n\in\N\}$ of $A_{h,\mu}$, with $(g_n,\lambda_n)\neq(g_m,\lambda_m)$ for all $n\neq m$. Without loss of generality we can 
   assume that $(g_n)_{n\in\N}$ is either constant, strictly decreasing or strictly increasing. Since the sequence $(g_n+\lambda_n)_{n\in\N}$ is constant, it is in particular decreasing, 
  so $(g_n)_n$ cannot be strictly increasing by Lemma \ref{sequencesLem}; since $\supp(a)$ is well-ordered, $(g_n)_n$ cannot be strictly decreasing. Therefore, 
  $(g_n)_n$ is constant, but then $\lambda_n$ must also be constant, which contradicts $(g_n,\lambda_n)\neq(g_m,\lambda_m)$ for all $n\neq m$. This proves that 
  $A_{h,\mu}$ is finite.    
   \end{proof}
   
   Thus, formula (\ref{defofder}) defines a map on $K$, and it is easy to see from its definition that it is a derivation.
     It remains to see if $(K,v,D)$ is a 
      pre-differential-valued field.
      
      \begin{Prop}\label{leadingcoefofDa}
       Let $a=\sum_{g\in G}a_gt^g\in K$ with $g=v(a)\neq0$ and  $\lambda=\psi(g)$.  Then 
       $v(D(a))\geq g+\lambda$, and the coefficient of $D(a)$ at $g+\lambda$ is 
 $a_gu_{\lambda}(g_{\lambda})$. In particular, if $u_{\lambda}(g_{\lambda})\neq0$ then 
      $v(D(a))=v(a)+\psi(v(a))$.     
      \end{Prop}
  \begin{proof}
   Let $h+\mu\in A=\{f+\nu\mid f\in\supp(a), \nu\in\supp(f)\}$ with $h+\mu\leq g+\lambda$. Since $g=v(a)$ and $h\in\supp(a)$ we have 
   $g\leq h$. By \ref{Fundlemsummability}, we have $\mu\in\supp(g)$, and since 
$\lambda=\psi(g)$ it follows that $\lambda\leq\mu$, hence $g+\lambda\leq h+\mu$. This proves that $g+\lambda\leq\min A$. 
Now note that if $g<h$ then $h+\mu\leq g+\lambda$ would imply $\mu<\lambda$ which would contradict $\mu\in\supp(g)$, so 
$h+\mu= g+\lambda$ implies $h=g$ and thus $\mu=\lambda$. It follows that $A_{g,\lambda}=\{(g,\lambda)\}$, and so by formula (\ref{defofder}) the coefficient in front of $t^{g+\lambda}$ is just
$a_gu_{\lambda}(g_{\lambda})$.
  \end{proof}

      If $u_{\lambda}(g_{\lambda})=0$, then $v(D(a))$ is not equal to $v(a)+\psi(v(a))$. It follows that in general,  $K$ endowed with 
      the derivation $D$ is not even a  pre-differential valued field since there may be constants whose valuation is not trivial. However, 
$(K,v,D)$ becomes a differential-valued field if we impose a condition on $u_{\lambda}$:
   \begin{Prop}\label{ugammainjProp}
    If $u_{\lambda}$ is injective for every $\lambda\in\Psi$,
    then $(K,v,D)$ is a differential-valued 
   field with asymptotic couple $(G,\psi)$.
   \end{Prop}

   \begin{proof}
    By Proposition \ref{leadingcoefofDa}, we have $v(D(a))=v(a)+\psi(v(a))$ for every $a\notin\unitsv$. 
    In particular, $v(a)\neq0$ implies $D(a)\neq0$. It follows that  $\constants=k.t^0$, so (DV1) is satisfied. Now let 
    $a,b\in K$ with 
    $v(a),v(b)>0$ and $b\neq0$. By Proposition \ref{leadingcoefofDa}, we have 
    $v(D(a))=v(a)+\psi(v(a))$ and $v(\frac{D(b)}{b})=\psi(v(b))$. By Axiom (AC3), it follows that 
   $v(D(a))>v(\frac{D(b)}{b})$. If $v(a)=0$, then we write $a=c+\epsilon$ with $c\in\constants$ and $\epsilon\in\idealv$ and we have 
  $D(a)=D(\epsilon)$ which brings us back to the case $v(a)>0$. This proves (DV2). The relation $v(D(a))=v(a)+\psi(v(a))$ proves that $(G,\psi)$ is 
the asymptotic couple associated to $(K,v,D)$.
   \end{proof}

    In the case where $k$ is ordered, we can even make $(K,v,D)$ a H-field under some additional assumption. Remember that, because 
    $(G_{\psi})_{\lambda}$ is convex in $(G,\leq)$, the order $\leq$ induces an order on $B_{\lambda}$ defined by 
    $g+(G_{\psi})_{\lambda}\leq h+(G_{\psi})_{\lambda}$ if and only if \newline$(g-h\in(G_{\psi})_{\lambda})\vee (g-h\notin(G_{\psi})_{\lambda}\wedge g\leq h)$.
   \begin{Prop}\label{ugammaopProp}
    Assume that $k$ is an ordered field and that $u_{\lambda}$ is an order-reversing embedding of groups for each $\lambda\in\Psi$.    
    Then $(K,\leq,D)$ is a H-field with asymptotic couple $(G,\psi)$.
   \end{Prop}

   \begin{proof}    
    We already know that $(K,v,D)$ is a differential-valued field and that \newline
$\ringv=\{a\in K\mid\exists c\in\constants, \vert a\vert\leq c\}$.
   Now assume that $a\in K$ is such that $a>\constants$, $a=\sum_{g\in G}a_gt^g$. It follows that, for $g:=v(a)$, $a_g>0$ and  $g<0$. 
   This implies that $g_{\lambda}<0$, where $\lambda=\psi(g)$. By Proposition \ref{leadingcoefofDa}, the sign of $D(a)$ is the sign of 
  $a_gu_{\lambda}(g_{\lambda})$. Since $u_{\lambda}$ is order-reversing, it follows that $a_gu_{\lambda}(g_{\lambda})>0$, hence $D(a)>0$. This proves that 
$(K,\leq,D)$ is a H-field.
   \end{proof}

    This solves our problem for the case where $(G,\psi)$ is maximally valued. Now we want to extend those results to the general case. 
    Assume then that $(G,\psi)$ is an arbitrary H-type asymptotic couple. Let $\divhull{G}$ denote the divisible hull of $G$, i.e. the group 
    $G\bigotimes_{\Z}\Q$. Up to isomorphism, $\divhull{G}$ is the smallest divisible abelian group containing $G$. 
    It is easy to see that both $\leq$ and $\psi$ extend uniquely to $\divhull{G}$, so that 
    $(\divhull{G},\psi)$ is also a H-type asymptotic couple.    
    Viewing $(\divhull{G},\psi)$ as a valued group, we can use Hahn's embedding theorem: 
   there exists an embedding of valued groups $\iota:(\divhull{G},\psi)\hookrightarrow (H:=\Hahn_{\lambda\in\Psi}\divhull{B}_{\lambda},\psi)$, 
   where $\divhull{B}_{\lambda}$ denotes the divisible hull of $B_{\lambda}$, and $\psi$ is defined on $H$ by 
   $\psi(h)=\iota(\min\supp(h))$.
    It remains to show that $H$ can be endowed with an order 
extending the one on $\divhull{G}$, so that $\iota$ is actually an embedding of asymptotic couples. For this, we use results from \cite{LehericyKuhlmann}. Since  
  $\psi$ is H-type, it is a coarsening of $v_G$, so it follows from \cite[Theorem  3.2]{LehericyKuhlmann} that
$\leq$ induces an order on each $\divhull{B}_{\lambda}$. By \cite[Theorem 4.6]{LehericyKuhlmann}, we can lift this family of orders $(\leq_{\lambda})_{\lambda\in\Psi}$ to 
$H$. This gives us a group order on $H$ whose restriction to $\iota(\divhull{G})$ is exactly the order of $\divhull{G}$. Now if we restrict 
$\iota$ to $G$, we have an embedding of asymptotic couples $\iota: (G,\psi)\to (H,\psi)$.
With this embedding $\iota$ of asymptotic couples, we can see 
$K$ as a subfield of $k((H))$: consider the embedding $\rho:\sum_{g\in G}a_gt^g\mapsto\sum_{g\in G}a_gt^{\iota(g)}$. Note that if $k$ is ordered   then this embedding 
preserves the order 
of $K$. Now we know from what we have done before that $k((H))$ can be endowed with a derivation $D$ given by the formula (\ref{defofder}). It is easy to see from this 
formula that $\rho(K)$ is stable under $D$, so $D$ is also a derivation on $\rho(K)$ which gives us a derivation on $K$. With Proposition \ref{leadingcoefofDa} we could 
show again that 
if each $u_{\lambda}$ is injective, then $(K,v,D)$ is a differential-valued field, and if $k$ is ordered and each $u_{\lambda}$ is an order-reversing embedding, then 
$(K,\leq,D)$ is a H-field.

    Thus, the method described above allows us to define a derivation on any field of generalized power series $k((G))$ 
    where $(G,\psi)$ is a given H-type asymptotic couple. However, if we want to have a differential-valued field, we saw that our method only works 
    if each $B_{\lambda}$ ($\lambda\in\Psi$) is embeddable into $(k,+)$ as a group. One can then wonder if we could find a method 
    which does not need this condition; the next proposition proves that it is not possible:

    \begin{Prop}\label{embeddingnecessaryProp}
     Let $(G,\psi)$ be a H-type asymptotic couple and $k$ a field. 
     Let $D$ be a derivation on $K:=k((G))$ such that $(K,v,D)$ is a differential-valued field with asymptotic couple
     $(G,\psi)$.  Then for each 
$\lambda\in\Psi$, there exists a group embedding $u_{\lambda}$ from $B_{\lambda}$ into $(k,+)$. If moreover $k$ is ordered and 
$(K,\leq,D)$ is a H-field, then we can even choose $u_{\lambda}$ so that it is order-reversing.     
    \end{Prop}
    
    \begin{proof}
      We start by showing the following claim:
      \begin{Claim*}
       Let $h,g\in G$ be such that  $\psi(h)=\psi(g)$ and $\psi(g-h)>\psi(g)$. Then  $D(t^g)$ and $D(t^h)$ have the same 
      leading coefficient.
      \end{Claim*}
      \begin{proof}
	By the product rule, we have $D(t^h)=t^{h-g}D(t^g)+t^gD(t^{h-g})$. Moreover, we have 
         $v(t^{h-g}D(t^g))=h-g+g+\psi(g)=h+\psi(g)$ and \newline
$v(t^gD(t^{h-g}))=g+h-g+\psi(h-g)=h+\psi(h-g)$. Since $\psi(g-h)>\psi(g)$, 
 we have   $v(t^{h-g}D(t^g))<v(t^gD(t^{h-g}))$. It follows that $v(D(t^h))=v(t^{h-g}D(t^g))$. Therefore, the leading coefficient of 
      $D(t^h)$ is the leading coefficient of $t^{h-g}D(t^g)$, which is equal to the leading coefficient of $D(t^g)$.
\end{proof}

        Now let $\lambda\in\Psi$. Set $u_{\lambda}(0):=0$. For any $g_{\lambda}\in B_{\lambda}$,
        take any $g\in \grouppsiplus$ such that $g+\grouppsiminus=g_{\lambda}$ and define
        $u_{\lambda}(g_{\lambda})$  as the leading coefficient  of $D(t^g)$. 
        The claim makes sure that $u_{\lambda}(g_{\lambda})$ does not depend on the choice of $g$, so this gives us a well-defined map 
	$u_{\lambda}:B_{\lambda}\to k$.
 One can easily check that this is a group 
homomorphism. The fact that $\ker u_{\lambda}=\{0\}$ follows from the fact that $D(a)\neq0$ when $v(a)\neq0$. Now assume that $k$ is ordered and that 
$(K,\leq,D)$ is a H-field. Let $\lambda\in\Psi$ and $h_{\lambda},g_{\lambda}\in B_{\lambda}$ with $h_{\lambda}<g_{\lambda}$. 
Take $g,h\in G^{\lambda}$ with $g+\grouppsiminus=g_{\lambda}$ and $h+\grouppsiminus=h_{\lambda}$. We then have 
$h<g$. It follows that $v(t^{h-g})<0$, hence $t^{h-g}>\constants$, which by (PH3) implies $D(t^{h-g})>0$, so the leading coefficient 
$u_{\lambda}(h_{\lambda}-g_{\lambda})=u_{\lambda}(h_{\lambda})-u_{\lambda}(g_{\lambda})$ of $D(t^{h-g})$ is positive, hence 
$u_{\lambda}(h_{\lambda})>u_{\lambda}(g_{\lambda})$, so $u_{\lambda}$ is order-reversing.      
\end{proof}

    We can now formulate our answer to Question \ref{question}. Note that the existence of an order-reversing group homomorphism between two groups is equivalent to the 
existence of an order-preserving homomorphism between them.
    \begin{Thm}\label{mainThm}
     Let 
      $k$ be a field (respectively, an ordered field) 
      and $(G,\psi)$ a H-type asymptotic couple. Let $(\Psi,(\comppsi)_{\lambda\in\Psi})$ be
      the skeleton of the valuation $\psi$. There exists a derivation $D$ on $k((G))$ making 
      $(k((G)),v,D)$ a differential valued field (respectively, a H-field) with asymptotic couple $(G,\psi)$
      if and only if each 
      $\comppsi$ is embeddable into $(k,+)$ (respectively, into $(k,+,\leq)$). In that case, $D$ can be defined 
      by the formula (\ref{defofder}).
    \end{Thm}
     \begin{proof}
      It follows directly from Propositions \ref{ugammainjProp}, \ref{ugammaopProp} and \ref{embeddingnecessaryProp}
     \end{proof}

\subsection{Fields of power series admitting a H-derivation}\label{Hderivationsection}

We now want to use Theorem \ref{mainThm} to answer Question \ref{question1} of the introduction.
 This means we need to characterize the ordered groups $G$ which can be endowed with a map $\psi$ satisfying the properties of Theorem \ref{mainThm}. 
 Note that if $(G,\psi)$ is a  H-type asymptotic couple, then $\psi$ is consistent with $v_G$, which means that $\psi$ naturally 
induces two maps on $\Gamma$: \begin{enumerate}
                              \item the map $\hat{\psi}:\Gamma\to G$ defined by $\hat{\psi}(v_G(g)):=\psi(g)$.
			      \item the map $\omega:\Gamma\to\Gamma\cup\{\infty\}$ defined by $\omega(\gamma):=v_G(\hat{\psi}(\gamma))$.
                             \end{enumerate}
The main idea to answer Question \ref{question1} is to characterize the maps on $\Gamma$ which can be lifted to a map $\psi$ satisfying the conditions of Theorem \ref{mainThm}. 
This is connected to the notion of shift. 
A  map on $\sigma: \Gamma\to\Gamma\cup\{\infty\}$ is called a
 \textbf{right-shift} if $\sigma(\gamma)>\gamma$ for all $\gamma\in\Gamma$.
The authors of \cite{KuhlmannMatusinski} already found a connection between shifts on $\Gamma$ and the existence of  
Hardy-type derivations on $K$. In particular, it was showed in \cite{KuhlmannMatusinski} that a shift on $\Gamma$ can be lifted to a derivation on $\R(G)$, where $G=\Hahn_{\gamma\in\Gamma}\R$. 
We show here that any H-derivation comes  from a shift on $\Gamma$ (see Theorems \ref{answertoquestion1} and \ref{HardytypeThm}).

  We extend the notion of shift to maps from $\Gamma$ to $G^{\leq0}$: we say that a map  $\sigma:\Gamma\to G^{\leq0}$ is a \textbf{right-shift} if the map $v_G\circ\sigma:\Gamma\to\Gamma\cup\{\infty\}$ is a right-shift.
The following two lemmas show the connection between asymptotic couples and shifts. The increasing right-shifts of $\Gamma$ are exactly the maps induced by 
asymptotic couples of cut point $0$:

\begin{Lem}\label{sigmacomesfrompsiLem}
     Let $\sigma:\Gamma\to G^{\leq0}$ (respectively, $\omega:\Gamma\to\Gamma\cup\{\infty\}$) be a map. Then the following statements are equivalent: 
      \begin{enumerate}[(i)]
       \item There exists a map $\psi:G^{\neq0}\to G$ such that $(G,\psi)$ is a H-type asymptotic couple with cut point $0$ and 
        $\sigma(v_G(g))=\psi(g)$ (respectively, $\omega(v_G(g))=v_G(\psi(g))$) for all $g\in G^{\neq0}$.
	\item $\sigma$ (respectively, $\omega$) is an increasing right-shift.
      \end{enumerate}     
    \end{Lem}
    \begin{proof}
     Assume that (i) holds. Then by (ACH), $\sigma$ must be increasing and since $0$ is a cut point then $\sigma$ must be a 
     right-shift. Moreover, by Lemma \ref{gapiscutpointLem}(ii)
      we have $\sigma(\Gamma)\subseteq G^{\leq0}$.
    It then follows that $\omega=v_G\circ\sigma$ is an increasing right-shift. Conversely, assume $\sigma$ is an increasing right-shift. Define 
    $\psi(g):=\sigma(v_G(g))$ for any $g\in G^{\neq0}$.
 (AC2) is obviously satisfied. Let $g,h\in G^{\neq0}$. We have $v_G(g-h)\geq\min(v_G(g),v_G(h))$, and since $\sigma$ is increasing it follows that 
$\sigma(v_G(g-h))\geq\min(\sigma(v_G(g)),\sigma(v_G(h)))$ which proves (AC1). Now assume $0<h$. Since $\sigma$ is a right-shift, we have 
$v_G(\psi(h))>v_G(h)$, so $h+\psi(h)$ must be positive. Since $\sigma(\Gamma)\subseteq G^{\leq0}$, $\psi(g)$ is not positive, so (AC3) holds. 
Finally, (ACH) follows directly from the fact that $\sigma$ is increasing. Now let $c$ be a cut point for 
$\psi$. If $c\neq0$, then Lemma \ref{lemmacutpoint}(ii) implies $v_G(c)=v_G(\psi(c))$, hence 
$v_G(\sigma(v_G(c)))=v_G(c)$, which contradicts the fact that 
$\sigma$ is a right-shift. Therefore, we must have $c=0$, and 
 $\psi$ satisfies the desired conditions.
  If we are only given an increasing right-shift $\omega:\Gamma\to\Gamma\cup\{\infty\}$, we can choose a 
  $g_{\delta}\in G^{<0}$ with $v_G(g_{\delta})=\delta$ for each $\delta\in \omega(\Gamma)$. We can then set 
  $\sigma(\gamma):=g_{\omega(\gamma)}$ for every $\gamma\in\Gamma$. This gives us an increasing right-shift $\sigma:\Gamma\to G^{\leq0}$ with 
  $\omega=v_G\circ\sigma$. We already proves that there exists $\psi$ with $0$ as a cut point and such that $\sigma(v_G(g))=\psi(g)$. This equality then 
  implies $\omega(v_G(g))=v_G(\psi(g))$.
    \end{proof}

  Moreover, there is a natural way of associating a shift to any H-type asymptotic couple:
    \begin{Lem}\label{shiftexistenceLem}
     Let $(G,\psi)$ be a H-type asymptotic couple. There exists an increasing right-shift \newline
      $\sigma:\Gamma\to G^{\leq0}$ such that for any $g,h\in G^{\neq0}$, 
    $\sigma(v_G(g))\leq\sigma(v_G(h))\Leftrightarrow \psi(g)\leq\psi(h)$.
    \end{Lem}

    \begin{proof}
     We make a case distinction following Lemma \ref{trichotomy}.
     Assume first that $\psi$ has a gap or a maximum 
    $c$. Set $\psi':=\psi-c$ and let $\sigma(v_G(g)):=\psi'(g)$.  By Lemma \ref{translatebygap}, $0$  is a cut point for $\psi'$, so it follows from 
    Lemma \ref{sigmacomesfrompsiLem} that $\sigma$ is an increasing right-shift. It is clear
 that $\sigma(v_G(g))\leq\sigma(v_G(h))\Leftrightarrow \psi'(g)\leq\psi'(h)$ holds. Since 
$\psi'(g)\leq\psi'(h)\Leftrightarrow \psi(g)\leq\psi(h)$, it follows that $\sigma$ has the desired properties. Now we just have to consider the case where 
$(G,\psi)$ has asymptotic integration. In that case, let  $\chi$ be the same map as in Proposition \ref{exprank}, i.e. 
$\chi(0):=0$, $\chi(g):=\int\psi(g)$ for all $g<0$ and $\chi(g):=-\chi(-g)$ for all $g>0$. Now define $\sigma$ by 
$\sigma(v_G(g)):=\chi(g)$. Since $\chi$ is a centripetal precontraction, it follows that $\sigma$ is an increasing right-shift. It is also clear 
from the definition of $\sigma$ that $\sigma(v_G(g))\leq\sigma(v_G(h))\Leftrightarrow \chi(g)\leq\chi(h)$. Moreover, we know from 
\cite[Proposition 2.3]{Aschdries2} that the map 
 $D_G$ is increasing, so $\int$ is also increasing, so
$\chi(g)=\chi(h)\Leftrightarrow \psi(g)=\psi(h)$ holds, which proves that $\sigma$ satisfies the desired properties. 
    \end{proof}

For any 
   increasing map $\sigma:\Gamma\to G^{\leq0}$ and any $f\in\sigma(\Gamma)$, we now set \newline
  $\groupsigplus:=\{g\in G\mid \sigma(v_G(g))\geq f\}\cup\{0\}$ and 
  $\groupsigminus:=\{g\in G\mid \sigma(v_G(g))>f\}\cup\{0\}$.  We have the following: 

\begin{Lem}\label{shiftsubgroupsLem}
 Let $\sigma:\Gamma\to G^{\leq0}$ be an increasing map. Then for any $f\in\sigma(\Gamma)$, the sets 
$\groupsigplus$ and $\groupsigminus$ are convex subgroups 
of $G$.
\end{Lem}

\begin{proof}
Just note that $\sigma\circ v_G$ is a valuation and a coarsening of $v_G$.
\end{proof}

 Lemma \ref{shiftsubgroupsLem} allows us to define the group $\compsig:=\groupsigplus/\groupsigminus$. Since 
 $\groupsigminus$ is convex, $\compsig$ is naturally an ordered group. 
This allows us to partially answer Question \ref{question1}:

\begin{Thm}\label{answertoquestion1}
     Let $G$ be an ordered abelian group and $k$ a field (respectively, an ordered field). Let $\Gamma$ denote the value chain of $G$. 
     Then there exists a H-derivation $D$ on $K:=k((G))$ making $(K,v,D)$ a  differential-valued field (respectively, a H-field) if and only if there exists 
     an increasing map $\sigma:\Gamma\to G^{\leq0}$ such that the following holds: 
	\begin{enumerate}[(1)]
	 \item The map $v_G\circ\sigma$ is  a right-shift
	  \item For any $f\in\sigma(\Gamma)$, $\compsig$ is embeddable into $(k,+)$ (respectively, $(k,+,\leq)$).
	\end{enumerate}
    \end{Thm}

 \begin{proof}
  Assume that $(K,v,D)$ is a differential-valued field having a H-type asymptotic couple $(G,\psi)$. By Theorem \ref{mainThm}, it follows that
$\comppsi$ is embeddable in $(k,+)$ for every $\lambda\in\Psi$. Now take $\sigma$ as in Lemma \ref{shiftexistenceLem}. Let
$f\in\sigma(\Gamma)$, $f=\sigma(v_G(g))$ for some $g\in G$. We have 
  $\sigma(v_G(h))\geq f\Leftrightarrow \sigma(v_G(h))\geq \sigma(v_G(g))\Leftrightarrow \psi(h)\geq \psi(g)$. It follows that 
 $\groupsigplus=\grouppsiplus$, where $\lambda=\psi(g)$. Similarly, $\groupsigminus=\grouppsiminus$, hence $\compsig=\comppsi$, so 
$\compsig$ is embeddable into $(k,+)$. If $(K,\leq,D)$ is a H-field then  by Theorem \ref{mainThm}, 
$\compsig$ must be embeddable as an ordered group in $(k,+,\leq)$. This proves one direction of the theorem, let us prove the converse. Assume that such a
$\sigma$ as in the theorem exists. By Lemma \ref{sigmacomesfrompsiLem}, there exists $\psi$ on $G$ making $(G,\psi)$ a H-type asymptotic couple and such that $\sigma$ is  induced by 
$\psi$. It follows that, for each $\lambda\in\Psi$, we have $\lambda\in\sigma(\Gamma)$,
$\grouppsiplus=G_{\sigma}(\lambda,+)$ and $\grouppsiminus=G_{\sigma}(\lambda,-)$. By assumption, it follows that each 
$\comppsi$ is embeddable into $(k,+)$, and the existence of $D$ then follows from Theorem \ref{mainThm}. If each 
$\compsig$ is embeddable into $(k,+,\leq)$, then so is each $\comppsi$, and we conclude by Theorem \ref{mainThm} again.
 \end{proof}

  \begin{Rem}
   If we are given a $\sigma$ as in Theorem \ref{answertoquestion1}, we can explicitly construct $D$. We first define $\psi$ on $G$ by 
  $\psi(g):=\sigma(v_G(g))$. This gives us a H-type asymptotic couple $(G,\psi)$. We then define $D$ with formula 
  (\ref{defofder}) above.
  \end{Rem}

      \subsection{Hardy-type derivations}\label{Hardytypesection}
       The goal of this section is to characterize fields of power series which can be endowed with a Hardy-type derivation as defined in \cite{KuhlmannMatusinski}. 
       \begin{Def}\label{Hardytypederivation}
If $(G,\psi)$ is an asymptotic couple, we say that $\psi$ is \textbf{of Hardy type} if $(G,\psi)$ is H-type and 
$\psi(g)=\psi(h)\Rightarrow v_G(g)=v_G(h)$ for all $g,h\in G^{\neq0}$. We say that a derivation $D$ on an ordered field $(K,\leq)$ is of Hardy type if 
$(K,\leq,D)$ is a $H$-field with asymptotic couple $(G,\psi)$ such that $\psi$ is of Hardy type.
\end{Def}
 The natural derivation of
a Hardy field is an example of a Hardy-type derivation. Note that if $\psi$ is Hardy-type, then we have $v_G(g)=v_G(h)\Leftrightarrow\psi(g)=\psi(h)$, which 
means that $\psi$ and $v_G$ are equivalent as valuations. In particular,  
the valued groups $(G,v_G)$ and $(G,\psi)$ have the same components.

In general, H-derivations are not necessarily Hardy-type derivations. However, the two notions coincide in a field of power series if the field of coefficients is 
archimedean: 
\begin{Prop}\label{Hderivationsforarchfields}
 Let $k$ be an archimedean ordered field, $G$ an ordered abelian group, $K:=k((G))$ and 
$D$ a derivation on $K$. Then $(K,\leq,D)$ is a H-field if and only if $D$ is of Hardy-type.
\end{Prop}

 \begin{proof}
  Assume $(K,\leq,D)$ is a H-field with asymptotic couple $(G,\psi)$. By Theorem 
  \ref{mainThm}, each component of the valued group $(G,\psi)$ is embeddable as an ordered group into 
 $(k,+,\leq)$. It follows that each component of the valued group $(G,\psi)$ is archimedean, and it follows that $\psi$ must be of Hardy type. Indeed, if $\psi$ is not 
of Hardy type, then there are $g,h\in G$ with $v_G(g)>v_G(h)$ and $\psi(g)=\psi(h)$. It follows that 
$g,h$ are not archimedean-equivalent. Now set $\lambda:=\psi(h)$. Then $g+\grouppsiminus$, $h+\grouppsiminus$ are two non-zero elements of 
$\comppsi$, but they are not archimedean-equivalent, which is a contradiction.
 \end{proof}

 We can now answer a variant of Question \ref{question1}, where $D$ is required to be of Hardy type. The criterion for the existence of $D$ in this case is 
 simpler than the criterion for the existence of a H-derivation given by Theorem \ref{answertoquestion1}. In particular, the groups $\compsig$ from Theorem \ref{answertoquestion1} do not appear anymore:
\begin{Thm}\label{HardytypeThm}
 Let $G$ be an ordered abelian group and $k$  an ordered field. Let $\Gamma$ denote the value chain of $G$. 
     Then there exists a Hardy-type derivation $D$ on $K:=k((G))$  if and only if the following conditions are satisfied: 
      \begin{enumerate}[(1)]
       \item Each $\compv$ is embeddable as an ordered group into $(k,+,\leq)$.
      \item There exists 
     a strictly increasing right-shift $\sigma:\Gamma\to G^{\leq0}$.
      \end{enumerate}
 
\end{Thm}

\begin{proof}
 Assume there exists a Hardy-type derivation $D$ on $K$, denote by $\psi$ the map induced by the logarithmic derivative on $G$. 
 By Theorem \ref{mainThm}, each $\comppsi$ is embeddable into $(k,+,\leq)$.  
 Since $D$ is Hardy-type, it follows that $\compv$ is embeddable into $(k,+,\leq)$ for every $\gamma\in\Gamma$ ($(G,v_G)$ and $(G,\psi)$ have the same components).
 This shows (1). Now take $\sigma$ as in Lemma \ref{shiftexistenceLem}. For all $g,h\in G^{\neq0}$, we have by definition 
 of $\sigma$: $\sigma(v_G(g))\leq\sigma(v_G(h))\Leftrightarrow \psi(g)\leq\psi(h)$, and  because 
 $\psi$ is Hardy-type we have  $v_G(g)\leq v_G(h)\Leftrightarrow\psi(g)\leq\psi(h)$. This shows that $\sigma$ is strictly increasing, hence (2).
 Conversely, assume that (1) and (2) hold.  
By Lemma \ref{sigmacomesfrompsiLem} there is $\psi$ on $G$ such that 
$\sigma$ is induced by $\psi$. Since  $\sigma$ is strictly increasing, it is in particular  injective, which implies that 
$\psi$ is Hardy-type. It then follows from (1) that $\comppsi$ is embeddable into $(k,+,\leq)$ for every 
$\lambda\in\Psi$. The existence of $D$ is then given by Theorem \ref{mainThm}. 
\end{proof}

In the case where $k$ contains $\R$, we can overlook condition (1) of Theorem \ref{HardytypeThm}:

\begin{Cor}
 Let $k$ be an ordered field containing $\R$ and $G$ an ordered abelian group. 
 Let $\Gamma$ denote the value chain of $G$. 
     Then there exists a Hardy-type derivation $D$ on $K:=k((G))$  if and only if there exists 
     a strictly increasing right-shift $\sigma:\Gamma\to G^{\leq0}$.
\end{Cor}

  \begin{proof}
   Since each $\compv$ is archimedean, and since $k$ contains $\R$, condition (1) of Theorem \ref{HardytypeThm} is always satisfied, so 
 $D$ exists if and only if condition (2) of Theorem \ref{HardytypeThm} is satisfied.
  \end{proof}

 Finally, if $k$ is archimedean, we can simplify Theorem \ref{answertoquestion1} to give a criterion for the existence of a H-derivation on $K$:
\begin{Cor}
 Let $k$ be an archimedean ordered field and $G$ an ordered abelian group with value chain $\Gamma$. There exists a derivation $D$ on $K:=k((G))$ making 
$(K,\leq,D)$ a H-field if and only if the following conditions are satisfied: 
      \begin{enumerate}[(1)]
       \item Each $\compv$ is embeddable as an ordered group into $(k,+,\leq)$.
      \item There exists 
     a strictly increasing right-shift $\sigma:\Gamma\to G^{\leq0}$.
  \end{enumerate}
  
\end{Cor}

\begin{proof}
 By Proposition \ref{Hderivationsforarchfields}, there exists a H-derivation on $K$ if and only if there exists a Hardy-type derivation on $K$. The claim then follows directly from 
 Theorem \ref{HardytypeThm}.
\end{proof}

\subsection{Realizing a linearly ordered set as a principal differential rank}\label{rankrealizedsection}
  Assume $(K,v,D)$ is a pre-differential valued field with principal differential rank $P$ and principal unfolded differential rank $Q$. We know by Proposition \ref{unfrankProp} that 
$P$ is either a principal final segment of $Q$ or equal to $Q$.
 The goal of this section is to show 
a converse statement, i.e that any pair $(P,Q)$ of linearly ordered sets, where $P$ is a principal final segment of $Q$ or $Q=P$, can be realized as the pair
``(principal differential rank, principal unfolded differential rank)'' of a certain  field of power series endowed with a Hardy-type derivation.

 The construction is done in three steps: we first show that any linearly ordered set $Q$ 
 can be realized as the principal $\omega$-rank of  a certain ordered set. We actually give an explicit example. 
We then show that there exists an asymptotic couple $(G,\psi)$ whose principal $\psi$-rank is $P$ and whose principal unfolded
$\psi$-rank is $Q$, where $P$ is any principal final segment of $Q$ or $Q$ itself. Finally, we use Theorem \ref{mainThm} to obtain the desired field.

 \begin{Ex}\label{examplegammaomega}
    We want to give an  explicit example of an ordered set $(\Gamma,\leq)$ with arbitrary principal $\omega$-rank.
  \begin{enumerate}[(a)]
   \item We first construct an example of principal rank 1. Take $\Gamma_1:=\Z$ ordered as usual an define $\omega_1(n):=n+1$. Then 
  it is easy to check that $\omega_1$ is a strictly increasing right-shift and that the principal $\omega_1$-rank of $(\Gamma_1,\leq)$ is $1$.
   \item We now construct an example which has principal $\omega$-rank $(Q,\leq)$, where $(Q,\leq)$ is an arbitrary linearly ordered set. Define 
    $\Gamma:=Q\times \Gamma_1$ and order $\Gamma$ as follows:\newline
    $(a,\gamma)\leq(b,\delta)\Leftrightarrow (a<^{\ast}b)\vee (a=b\wedge \gamma\leq\delta)$, where $<^{\ast}$ denotes the 
    reverse order of $<$. Now define $\omega(a,n):=(a,\omega_1(n))$. 
    One sees easily that $\omega$ is a strictly increasing right-shift.  Now consider the map: 
    $a\mapsto \{(b,n)\in\Gamma\mid a\leq b, n\in\Z\}$ from $Q$  to the set of final segments of $(\Gamma,\leq)$. One can see that 
    this is a strictly increasing bijection from $(Q,\leq)$ to the principal $\omega$-rank of $(\Gamma,\leq)$, so $(Q,\leq)$ is the 
    principal $\omega$-rank of $(\Gamma,\leq)$.
  \end{enumerate}

 \end{Ex}

\begin{Prop}\label{acwithgivenrankProp}
 Let $Q$ be a totally ordered set and $P$ a final segment of $Q$ such that $P$ is either equal to $Q$ or is a principal final segment of 
 $Q$. Then there exists a H-type asymptotic couple 
 $(G,\psi)$ whose principal $\psi$-rank is $P$ and whose principal unfolded $\psi$-rank is $Q$. Moreover, we can choose 
$(G,\psi)$ so that $\psi$ is of Hardy type.
\end{Prop}

\begin{proof}
 Let $(\Gamma,\leq)$ be an ordered set with a strictly increasing shift $\omega:\Gamma\to\Gamma$ such that the principal $\omega$-rank of $(\Gamma,\leq)$ is $Q$ 
(take for instance Example \ref{examplegammaomega}(b)). Let $G:=\Hahn_{\gamma\in\Gamma}\R$.
We know from Lemma  \ref{sigmacomesfrompsiLem} that there exists a map $\psi_0$ on $G^{\neq0}$ such that $(G,\psi_0)$ is a H-type asymptotic couple  with cut point $0$  and such that 
$\omega$ is the map induced by $\psi_0$ on $\Gamma$. Since $\omega$ is injective, it follows that $\psi$ is of Hardy type.
By proposition \ref{unfrankProp}(iv), the principal unfolded  $\psi_0$-rank of $(G,\psi_0)$ is equal to its 
principal  $\psi_0$-rank which is equal to $Q$ by Theorem \ref{diffrankthreelevelsprop}. If $Q=P$, then set 
$\psi:=\psi_0$, and then $(G,\psi)$ satisfies the condition we wanted. Now assume $Q\neq P$.
We know that $P$ is a principal final segment of $Q$, which means that there is  a $\psi_0$-principal convex subgroup $H$ of $G$ such that 
$P$ is  the set of $\psi_0$-principal convex subgroups of $G$ containing $H$. Now let 
$c\in H$ be such that $H$ is $\psi_0$-principal generated by $c$ and set $\psi(g):=\psi_0(g)+c$. 
Obviously, $\psi$ is Hardy-type. Now note that $c$ is either a gap or a maximum for $\psi$: indeed, by Lemma \ref{gapiscutpointLem}(ii), 
 we have for all $g>0$:
$\psi_0(g)\leq 0<\psi_0(g)+g$, hence $\psi_0+c\leq c<\psi_0(g)+c+g$. It then follows from 
Proposition \ref{unfrankProp}(v) that $Q$ is the principal unfolded $\psi$-rank of $G$. 
Moreover, we know from Proposition \ref{unfrankProp}(ii) that the principal $\psi$-rank of $G$ is the principal final segment of $Q$ generated by $H$, so it is $P$.
\end{proof}

    We can now state our Theorem:

   \begin{Thm}\label{rankrealizedThm}
     Let $Q$ be a totally ordered set and $P$ a final segment of $Q$ such that $P$ is either equal to $Q$ or is a principal final segment of 
 $Q$. Then there exists an ordered field  $k$, an ordered abelian group $G$ and a Hardy-type derivation $D$ on $K:=k((G))$ such that
  $(K,\leq,D)$ is a H-field of principal differential rank $P$ and of principal 
unfolded differential rank $Q$.
   \end{Thm}

   \begin{proof}
    By Proposition \ref{acwithgivenrankProp},  there exists a H-type asymptotic couple 
    $(G,\psi)$ which has principal $\psi$-rank $P$ and principal unfolded $\psi$-rank 
    $Q$, and such that $\psi$ is Hardy-type. Now set 
    $K:=\R((G))$.  Since $\psi$ is Hardy-type, each $\comppsi$ is archimedean, so it is embeddable into $(\R,+,\leq)$. By theorem \ref{mainThm}, there exists a derivation $D$ on 
    $K$ making $(K,\leq,D)$ a H-field with asymptotic couple $(G,\psi)$. By Theorem \ref{diffrankthreelevelsprop}, the principal differential rank of $(K,\leq,D)$ is equal to 
    $P$ and by definition the principal unfolded differential rank of $(K,\leq,D)$ is equal to $Q$.       
   \end{proof}

 \begin{Rem}\label{rankremark}
 \begin{enumerate}[(i)]
  \item 
The differential rank of $(K,v,D)$ is  completely determined up to isomorphism by 
its principal differential rank. Similarly, 
the unfolded differential rank of $(K,v,D)$ is  completely determined up to isomorphism by 
its principal unfolded differential rank (see Remark \ref{ordercompletionrem}).
If $P$ and $Q$ are as in Theorem \ref{rankrealizedThm}, then their order completions are realized as the 
differential rank, respectively as the unfolded differential rank of $(K,v,D)$.
\item   It would be interesting to improve Theorem \ref{rankrealizedThm} by requiring $K$ to have asymptotic integration. 
 However, our construction in Proposition \ref{acwithgivenrankProp} gives us an asymptotic couple with a gap. 
 In \cite[Section 4]{Gehret}, Gehret gave methods to extend a given asymptotic couple with a gap into an asymptotic couple with 
 asymptotic integration. However, this construction changes the unfolded differential rank, so we cannot use it in 
 Theorem \ref{rankrealizedThm}. It is unknown if Theorem \ref{rankrealizedThm} remains true if we require the field to have asymptotic integration.
 \end{enumerate}

 \end{Rem}


\begin{thebibliography}{12}
 \bibitem{Asch} M. Aschenbrenner
		    \emph{Some remarks about asymptotic couples},
		    Valuation theory and its applications, Vol.
II (Saskatoon, SK, 1999), Fields Inst. Commun., vol. 33, Amer. Math. Soc., Providence, RI, pp. 7–18.
MR 2018547 (2004j:03043), 2003
 \bibitem{Aschdries}M. Aschenbrenner and 
		    L. v.d.Dries \emph{Closed asymptotic couples},
		    Journal of Algebra 225, 309–358, 2000
  \bibitem{Aschdries2}M. Aschenbrenner and 
		    L. v.d.Dries \emph{H-fields and their Liouville extensions},
		    Math. Z. 242, 543–588, 2002
		    \bibitem{Aschdries3}M. Aschenbrenner and 
		    L. v.d.Dries \emph{Liouville-closed H-fields},
		    Journal of Pure and Applied Algebra 197, 83 – 139, 2002
\bibitem{DriesAschHoev}M. Aschenbrenner, L. v.d.Dries and J. v.d.Hoeven, \emph{Asymptotic Differential Algebra and Model Theory of Transseries}, 
Ann. of Math. Studies vol. 195, Princeton University Press, Princeton 2017
\bibitem{EnglerPrestel} A. Engler, A. Prestel, \emph{ Valued Fields}, Springer, 2005
\bibitem{Fuchs} L. Fuchs, \emph{Partially ordered algebraic systems}, Pergamon Press, 1963
\bibitem{Gehret} A. Gehret, \emph{The asymptotic couple of the field of logarithmic transseries}, 
Journal of Algebra 470, 1-36, 2017. 
 \bibitem{Hardy} G. Hardy, \emph{Orders of Infinity}, 2nd ed., Cambridge Univ. Press, London/New York, 1924
 \bibitem{Kuhlmanncontraction} F.-V. Kuhlmann, \emph{Abelian groups with contractions I},
                    Proceedings of the Oberwolfach
Conference on Abelian Groups 1993, Amer. Math. Soc. ContemporaryMathematics 171,
217 241, 1994
 \bibitem{FVKuhlmanndifffields} F.-V. Kuhlmann \emph{Maps on ultrametric spaces, Hensel's Lemma, and differential equations over
valued fields}, Comm. in Alg.39, 1730-1776, 2011
 \bibitem{Kuhlmann} S. Kuhlmann, \emph{Ordered exponential fields}, The Fields Institute Monograph Series, vol 12, 2000
  \bibitem{LehericyKuhlmann} S. Kuhlmann, G. Lehéricy,\emph{ A Baer-Krull theorem for quasi-ordered groups},
    Order, 2017, https://doi.org/10.1007/s11083-017-9432-5
 \bibitem{KuhlmannMatusinski} S. Kuhlmann and M. Matusinski, \emph{Hardy 
                   type derivations on generalized series fields}, 
                   Journal of Algebra 351 , 185-203, 2012
\bibitem{KuhlmannMatusinski2} S. Kuhlmann and M. Matusinski, \emph{Hardy 
                   type derivations on fields of exponential logarithmic series}, 
                   Journal of Algebra 345, 171-189, 2011
  \bibitem{KuhlmannPointMatusinski} S. Kuhlmann, M. Matusinski and F. Point, 
  \emph{The valuation difference rank of a quasi-ordered difference field}, 
   Groups, Modules and Model Theory - Surveys and Recent Developments
in Memory of Rüdiger Göbel, Springer Verlag, p399-414, 2017
\bibitem{Rosenlicht} M. Rosenlicht, \emph{Differential valuations}, Pacific J. Math. 86, 301–319, 1980
\bibitem{Rosenstein} J.G. Rosenstein, \emph{Linear Orderings}, Pure and Applied Mathematics, vol. 98,
Academic press, New York, 1982
\bibitem{ZariskiSamuel} O. Zariski, P. Samuel, \emph{Commutative Algebra II}, The University Series in Higher Mathematics, 1960
 \end{thebibliography}
\end{document}